\documentclass{svproc}
\usepackage{amsmath, verbatim, amsfonts, amssymb, color}
\usepackage{mathrsfs}
\usepackage{dsfont}

\renewcommand\labelenumi{\textup{\alph{enumi})}}
\renewcommand\theenumi\labelenumi

 \makeatletter
\def\@makefnmark{\hbox{(\@textsuperscript{\normalfont\@thefnmark})}}
\makeatother

\newtheorem{thm}{Theorem}
\newtheorem{defi}{Definition}
\newtheorem{lem}{Lemma}
\newtheorem{prop}{Proposition}

\newtheorem{exam}{Example}

\newtheorem{rem}{Remark}
\allowdisplaybreaks

\newcommand{\bone}{\mathbf{1}}
\newcommand{\N}{\mathbb{N}}
\newcommand{\R}{\mathbb{R}}
\newcommand{\C}{\mathbb{C}}

\newcommand{\T}{\mathbb{T}}

\newcommand{\bE}{{\bf E}}
\newcommand{\bP}{{\bf P}}
\newcommand{\m}{\mathfrak{m}}
\renewcommand{\d}{{\rm d}}

\makeatletter
\@namedef{subjclassname@2020}{%
  \textup{2020} Mathematics Subject Classification}
\makeatother

\providecommand{\ack}[1]{\par\addvspace\baselineskip
\noindent\ackname\enspace\ignorespaces#1}%
\def\subjclassname{\textup{2020} \textit{Mathematics Subject Classification:}}%
\providecommand{\subjclass}[1]{\par\addvspace\baselineskip
\noindent\subjclassname\enspace\ignorespaces#1}%

\begin{document}
\mainmatter
\title{\bfseries Equivalence of the strong Feller properties of analytic semigroups and associated resolvents}
\titlerunning{Equivalence of strong Feller properties}

\author{Seiichiro Kusuoka, Kazuhiro Kuwae and Kouhei Matsuura}
\authorrunning{S.~Kusuoka, K.~Kuwae and K.~Matsuura}

\tocauthor{David Berger, Franziska K\"uhn and Ren\'e L. Schilling}

\institute{Graduate School of Science, Kyoto University,\\
Kitashirakawa-Oiwakecho, Sakyo-ku, Kyoto 606-8502, Japan\\
\email{kusuoka@math.kyoto-u.ac.jp}\\
\quad
\\
Department of Applied Mathematics, Faculty of Science, Fukuoka University,\\
8-19-1 Nanakuma, Jonan-ku, Fukuoka 814-0180, Japan\\
\email{kuwae@fukuoka-u.ac.jp}\\
\quad
\\
Institute of Mathematics, University of Tsukuba,\\
1-1-1, Tennodai, Tsukuba, Ibaraki, 305-8571, Japan\\
\email{kmatsuura@math.tsukuba.ac.jp}  }

\maketitle

\begin{abstract} 
In this paper, we give sufficient conditions for the equivalence between semigroup strong Feller property and resolvent strong Feller property. 
\ack{This work was supported by JSPS KAKENHI Grant number 17K14204, 17H02846, and 20K22299.}
\end{abstract}

\noindent

\section{Introduction}\label{sec:intro}

The notion of the Feller property  was initiated by William Feller~\cite{Feller1952}. In the paper, a pair of one-dimensional parabolic diffusion equations is exhaustively studied through the associated semigroups, where we can observe the origin of the present Feller property. In \cite{DynkinBook}, the Feller property is defined for  Markov processes on compact metric spaces, which states that the associated semigroups map the family of continuous functions on the state space into itself. Later, this notion was extended beyond the compactness of the state space. The semigroup $\{P_t\}_{t \ge 0}$ of a Markov process on a locally compact separable metric space $E$ is now said to have the Feller property if each $P_t$ leaves invariant the family $C_{\infty}(E)$ of continuous functions on $E$ vanishing at infinity.  It is known that any Feller semigroup generates  a Markov process on the state space with strong Markov property and c\`adl\`ag path, which is called the \emph{Feller process}. An important subclass is formed by the strong Feller processes initiated by Girsanov~\cite{GirSF}. They are generated by transition semigroups with the strong Feller property, i.e., each $P_t$ maps the family 
$\mathscr{B}_b(E)$ of bounded Borel functions on $E$ into the family $C_b(E)$ of 
bounded continuous functions on it. In fact, Markov processes other than Feller processes can have such property, and they are also called strong Feller processes. For example, it is known that non-degenerate diffusion processes on 
 Riemannian manifolds are strong Feller processes (see Molchanov~\cite{Molchanov}). 

The strong Feller property is one of fundamental notions for  Markov processes, and has been studied from various sources.
For example, the transition semigroups of Markov processes with strong Feller property often possess density functions (with respect to canonical measures), and some potential theoretic  aspects of the process are studied. One remarkable result is that the concepts of polar sets and semi-polar sets coincide with each other for such Markov processes associated with semi-Dirichlet forms. See \cite[\S4, Theorems~4.1.2 and 4.2.7]{FOT} and \cite[\S3.5, Theorem~3.5.4]{Oshima} for the precise statement. It is also important to point out that the strong Feller property is used to determine the uniqueness of invariant measures for Markov processes (see, e.g., \cite[Section~11.3.2]{DZ}). Hence, the strong Feller property plays a crucial role for the ergodic theory of Markov processes. 

Although the definition of the strong Feller property stated above is for semigroups, this is also defined for resolvents in a natural way (Definition~\ref{def:SF}). It is well-known that the semigroup strong Feller property (SF) implies the resolvent strong Feller property (RSF).
As this fact suggests, it is often easier to confirm (RSF) than (SF).
Then, the question of under what conditions (RSF) means (SF) naturally arises. This is not obvious. In fact, the uniform motion to the right satisfies (RSF), but not (SF). See \cite[Remark~1.1~(1)]{KKT} for details. This is in contrast to the fact that the Feller property for a resolvent kernel is equivalent to the Feller property for the associated transition semigroup kernel (see \cite[Section~1]{KKT}). However, the uniform motion to the right is a little extreme example of non-symmetric Markov processes. Therefore, under an appropriate framework, (RSF) can imply (SF). 

In the present paper, we study the question stated above, and provide several sufficient conditions for it. For a given Markov process, we consider a situation in which the semigroup is extended to an analytic semigroup on some $L^p$-space with $p \in [1,+\infty)$. Then, we utilize the theory of analytic semigroups to describe a general conditions that strengthen (RSF) to (SF), which is the first main result of this paper (Theorem~\ref{thm:equi}). We also show that the semigroup of the   uniform motion to the right is not extended to an analytic semigroup on any $L^p$-space with respect to the invariant measure. Theorem~\ref{thm:equi} can be applied to Hunt processes associated with lower bounded semi-Dirichlet forms (Theorem~\ref{thm:app}). As a result, we find that (SF) can be obtained mainly from the assumptions of (RSF) and a kind of ultracontractivity of the resolvent. Even if the ultracontractivity is replaced with the ultracontractivity of the semigroup, the same conclusion holds (Theorem~\ref{thm:bkk}). However, we think Theorem~\ref{thm:bkk} does not follow immediately from Theorem~\ref{thm:app}. These theorems may be restrictive in that the semigroups  of Ornstein-Uhlenbeck processes have the strong Feller property, but  does not satisfy the ultracontractivity.  Because of this fact, we introduce the notion of local ultracontractivity, and extend the Theorems (Theorem~\ref{thm:loc}). Here, it is also necessary to study (RSF) for part processes of a given Markov process.

The organization of this paper is as follows. In Section~\ref{sec:stoc}, we prepare 
several related notions, for example, the Feller, and strong Feller properties of 
transition semigroup and resolvent kernels,  
analyticity of strongly continuous semigroup on $L^p$-spaces, lower bounded semi-Dirichlet forms and so on. 
In Section~\ref{sec:equi}, we establish a general criterion from the strong Feller property of resolvents to that for semigroups in terms of the analyticity of the semigroup on $L^p$-spaces. 
In Section~\ref{sec:appSemiDir}, we apply the result in the framework of 
lower bounded semi-Dirichlet forms. We also provide some examples to clarify when our results are effective.

\vspace{0.3cm}

\noindent
{\it Notation.} The following symbols and conventions are used in the paper.
\begin{itemize}
\item For $p \in [1,+\infty]$ and a measure space $(E,\mathscr{E},\mu)$, we denote by $L^p(E;\,\mu)$ the $L^p$-space on it. 
For $f \in L^p(E;\,\mu)$, we set $\|f\|_{L^p(E;\,\mu)}=\{\int_{E}|f(x)|^p\,\mu(\d x)\}^{1/p}$.   For $p,q\in [1,+\infty]$ and a bounded linear operator $T$ from $L^p(E;\,\mu)$ to $L^q(E;\,\mu)$, we denote by $\|T\|_{L^p(E;\,\mu)\to L^q(E;\,\mu)}$ the operator norm.
If $(E,\mathscr{E},\mu)$ is clear from the context, we simply write $\|T\|_{p \to q}$ in stead of $\|T\|_{L^p(E;\,\mu)\to L^q(E;\,\mu)}.$

\item For a topological space $E$, we denote by $\mathscr{B}(E)$ the Borel $\sigma$-algebra on $E$. We set
\begin{align*}
C(E)&:=\{u\mid  u\text{ is a real valued continuous function on $E$} \},\\
C_{0}(E)&:=\{ u \in C(E) \mid \text{the closure of }u^{-1}(\R \setminus \{0\}) \text{ in $E$ is compact}\},\\
\mathscr{B}(E)&:=\{ u\mid  u\text{ is a Borel measurable $[-\infty,+\infty]$-valued function on $E$}\},\\
\mathscr{B}_{+}(E)&:=\{ u\in \mathscr{B}(E) \mid u \text{ is $[0,+\infty]$-valued}\},\\
\mathscr{B}_{b}(E)&:=\{ u\in \mathscr{B}(E) \mid \|u\|_{\infty}<\infty\}.
\end{align*}
Hereafter, $\|u\|_{\infty}:=\sup_{x \in E}|u(x)|$ for $u \colon E \to \R$. When $E$ is a locally compact separable metric space, we denote by $C_\infty (E)$ the completion of $C_{0}(E)$ under $\|\cdot\|_{\infty}$.
\item We denote by ${\rm i}:=\sqrt{-1}$ the imaginary unit. 
The real and imaginary parts of $z\in \mathbb{C}$ are denoted by $\text{\rm Re}\,z$ and $\text{\rm Im}\,z$, respectively.
\item We set $\inf \emptyset=\infty$.
\item For $a,b\in\R$, we write $a\lor b:=\max\{a,b\}$ and $a\land b:=\min\{a,b\}$. 
\end{itemize}

\section{Preliminaries}\label{sec:stoc}

Let $(E,d)$ be a locally compact separable metric space, and let $E_{\partial}=E\cup \{ \partial \}$ be the one-point compactification. 
Let ${\bf X}=(\{X_t\}_{t \in [0,+\infty]},\{\bP_x\}_{x \in E_{\partial}})$ be a 
Hunt 
process  on $E$. That is, ${\bf X}$ is a right continuous process on $E$ with strong Markov property and satisfies the right continuity of sample paths on $[0,+\infty)$ 
and the existence of left limits in $E_{\partial}$ of sample paths on $(0,+\infty)$ (see \cite[Definition~A.1.23]{CF}). 
Define the transition semigroup of ${\bf X}$  by
\begin{align*}
P_tf(x):=\bE_{x}[f(X_t)],\quad x \in E;\,\, t \ge 0,\,\,f \in \mathscr{B}_b(E), 
\end{align*}
where $\bE_x$ denotes the expectation under $\bP_x$. The resolvent of ${\bf X}$ is defined by 
\begin{align*}
R_{\alpha}f(x)=\int_{0}^\infty e^{-\alpha t}P_tf(x)\,\d t,\quad x \in E;\,\, f \in \mathscr{B}_b(E),\,\, \alpha>0.
\end{align*}

We first formulate the strong Feller property of semigroups and resolvents.
\begin{defi}\label{def:SF}
{\rm 
\begin{enumerate}
\item
The semigroup $\{P_t\}_{t \ge 0}$ is said to have the strong Feller property if for any $f \in \mathscr{B}_b(E)$ and $t>0$, $P_tf$ is bounded continuous on $E$.
\item The resolvent $\{R_\alpha\}_{\alpha>0}$ is said to have the strong Feller property if for any $f \in \mathscr{B}_b(E)$ and $\alpha>0$, $R_{\alpha}f$ is  bounded continuous on $E$.
\end{enumerate}
}
\end{defi}
It is easy to see that the semigroup strong Feller property implies the resolvent strong Feller property.  Next, we introduce the $C_b$-Feller property, which is a  weaker concept of the strong Feller property.

\begin{defi}\label{def:CBFeller}
{\rm 
\begin{enumerate}
\item
The semigroup $\{P_t\}_{t \ge 0}$ is said to have the \emph{$C_b$-Feller property} if for any $f \in C_b(E)$ and $t>0$, $P_tf$ is bounded continuous on $E$.
\item The resolvent $\{R_\alpha\}_{\alpha>0}$ is said to have the $C_b$-Feller property if for any $f \in C_b(E)$ and $\alpha>0$, $R_{\alpha}f$ is  bounded continuous on $E$.
\end{enumerate}
}
\end{defi}

The Feller properties of semigroups and resolvents are defined as follows.
\begin{defi}\label{def:Feller}
{\rm
\begin{enumerate}
\item
The semigroup $\{P_t\}_{t\ge 0}$ is said to have the \emph{Feller property} if for any $f \in C_{\infty}(E)$ and $t>0$, $P_tf$ belongs to $C_\infty(E)$, and $\lim_{t \to 0}\|P_tf-f\|_{\infty}=0.$
\item The resolvent $\{R_\alpha\}_{\alpha >0}$ is said to have the \emph{Feller property} if for any $f \in C_{\infty}(E)$ and $\alpha>0$, $R_\alpha f$ belongs to $C_\infty(E)$, and $\lim_{\alpha \to \infty}\|\alpha R_{\alpha}f-f\|_{\infty}=0.$
\end{enumerate}
}
\end{defi}
By the same argument as in \cite[Section~1]{KKT}, we see that $\{P_t\}_{t \ge 0}$ possesses the Feller property if and only if so does $\{R_{\alpha}\}_{\alpha >0}$. We refer the reader to \cite[Proposition~3.1]{Az} or \cite[Proposition~3.1]{T} for probabilistic characterizations of the Feller property under the (resolvent) 
strong Feller property. 

Hereafter, we consider the following two assumptions: 
\begin{itemize}
\item[{\bf (A1)}:] There exists a positive Radon measure $\m$ on $E$  with full support such that 
\begin{align}
\int_EP_tf(x)\,\m(\d x)\leq\int_Ef(x)\, \m(\d x),\quad t\geq0,\;f\in\mathscr{B}_+(E).\label{eq:exc}
\end{align}

\item [{\bf (A2)}:] There exists a lower bounded regular semi-Dirichlet form on $L^2(E;\,\m)$ associated with ${\bf X}$, where $\m$ is a positive Radon measure having full topological support. 
\end{itemize}

Let $(\mathscr{E},\mathscr{F})$ be a lower bounded semi-Dirichlet form on $L^2(E;\,\m)$. Here $(\mathscr{E},\mathscr{F})$ is said to be a \emph{lower bounded semi-Dirichlet form on $L^2(E;\,\m)$} if $\mathscr{F}$ is a dense linear subspace of $L^2(E;\,\m)$ and $\mathscr{E}:\mathscr{F}\times\mathscr{F}\to\R$ is a closed bilinear form in the following sense $(\mathscr{E}\!.\,1)$, $(\mathscr{E}\!.\,2)$ and $(\mathscr{E}\!.\,3)$, and satisfies the semi-Dirichlet property $(\mathscr{E}\!.\,4)$:  
\begin{itemize}
\item[$(\mathscr{E}\!.\,1)$:]
There exists a non-negative constant $\alpha_0$ such that 
\[
\mathscr{E}_{\alpha_0}(u,u):=\mathscr{E}(u,u)+\alpha_0(u,u)_{L^2(E;\,\m)}\geq0\quad\text{ for \; all }\quad u\in\mathscr{F}.
\]

\item[$(\mathscr{E}\!.\,2)$:]
$\mathscr{E}$ satisfies the \emph{{\rm(}strong\,{\rm)} sector condition}: there exists a constant 
$K\geq1$ such that 
\[
|\mathscr{E}(u,v)|\leq K\mathscr{E}_{\alpha_0}(u,u)^{1/2}\mathscr{E}_{\alpha_0}(v,v)^{1/2}\quad\text{ for \; all }\quad u,v\in\mathscr{F}, 
\]
where $\alpha_0$ is the non-negative constant specified in $(\mathscr{E}\!.\,1)$.

\item[$(\mathscr{E}\!.\,3)$:]
$\mathscr{F}$ is a Hilbert space relative to the inner product 
\[
\mathscr{E}_{\alpha}^{(s)}(u,v):=\frac12\left(\mathscr{E}_{\alpha}(u,v)+\mathscr{E}_{\alpha}(v,u) \right)\quad\text{ for \; all }\quad \alpha>\alpha_0,
\]
where $\alpha_0$ is the non-negative constant specified in $(\mathscr{E}\!.\,1)$.

\item[$(\mathscr{E}\!.\,4)$:]
for all $u\in\mathscr{F}$ and $a\geq0$, $u\land a\in\mathscr{F}$ and $\mathscr{E}(u\land a,u-u\land a)\geq0$.

\end{itemize}

Under $(\mathscr{E}\!.\,2)$, we can deduce the following \emph{{\rm(}weak\,{\rm)} sector condition}: 
For $\alpha>\alpha_0$, 
\begin{align}
|\mathscr{E}_{\alpha}(u,v)|\leq K\mathscr{E}_{\alpha}(u,u)^{1/2}\mathscr{E}_{\alpha}(v,v)^{1/2}\quad\text{ for \; all }\quad u,v\in\mathscr{F},\label{eq:WSC}
\end{align}
where $K(\geq1)$ is the constant appeared in $(\mathscr{E}\!.\,2)$. 
Remark that \eqref{eq:WSC} is a stronger form of the weak sector condition stated in 
\cite[\S~1.1, (1.1.3)]{Oshima}.  
\vspace{0.3cm}

Under $(\mathscr{E}\!.\,1)$, $(\mathscr{E}\!.\,2)$ and $(\mathscr{E}\!.\,3)$, we see from \cite[Chapter~I, Theorems~1.12 and 2.8]{MR} that
$(\mathscr{E},\mathscr{F})$ admits strongly continuous semigroups 
$\{T_t\}_{t\geq0}$ and $\{T_t^*\}_{t\geq0}$ on $L^2(E;\,\m)$ such that  $\|T_t\|_{L^2(E;\,\m)\to L^2(E;\,\m)}\leq e^{\alpha_0 t}$, $\|T_t^*\|_{L^2(E;\,\m)\to L^2(E;\,\m)}\leq e^{\alpha_0 t}$, 
\[
(T_tf,g)_{L^2(E;\,\m)}=(f,T_t^*g)_{L^2(E;\,\m)}.
\] 
Hereafter, $(\cdot,\cdot)_{L^2(E;\,\m)}$ denotes the $L^2$ inner product with respect to $\m$. That is, $\{T_t^*\}_{t \ge 0}$ is the dual semigroup of $\{T_t\}_{t \ge 0}.$
For $\alpha>\alpha_0$ and $f \in L^2(E;\,\m)$, we define 
$G_{\alpha}f=\int_0^{\infty}e^{-\alpha t}T_tf\d t$ and 
$G_{\alpha}^*f=\int_0^{\infty}e^{-\alpha t}T_t^*f\d t$,   the integrals being defined as the Bochner integral in $L^2(E;\,\m)$. It then follows from \cite[Chapter~I, Proposition~1.10 and Theorem~2.13]{MR} that
\[
\hspace{-1cm}\mathscr{E}_{\alpha}(G_{\alpha}f,u)=(f,u)_{L^2(E;\,\m)}=\mathscr{E}_{\alpha}(u,G_{\alpha}^*g),
\]
 for all $f\in L^2(E;\,\m)$, $u\in\mathscr{F}$, and $\alpha>\alpha_0$.
The resolvents $\{G_{\alpha}\}_{\alpha>\alpha_0}$ and $\{G_{\alpha}^*\}_{\alpha>\alpha_0}$ are strongly continuous on $L^2(E;\,\m)$ in the sense that $\lim_{\alpha\to\infty}\alpha G_{\alpha}f=\lim_{\alpha\to\infty}\alpha G_{\alpha}^*f=f$ in $L^2(E;\,\m)$. 

\vspace{0.3cm}

The condition $(\mathscr{E}\!.\,4)$ is equivalent to the following conditions $(\mathscr{E}\!.\,4c)$, 
 or $(\mathscr{E}\!.\,4d)$: 
\begin{itemize}
\item[$(\mathscr{E}\!.\,4c)$:]
$\{T_t\}_{t\geq0}$ is sub-Markov: If $f\in L^2(E;\,\m)$ satisfies $0\leq f\leq 1$ $\m$-a.e., then $0\leq T_tf\leq1$ $\m$-a.e.

\item[$(\mathscr{E}\!.\,4d)$:] $\{T_t^*\}_{t\geq0}$ is positivity preserving and contractive in $L^1(E;\,\m)$: If $f\in L^1(E;\,\m)$ satisfies $f\geq0$ $\m$-a.e., then $T_t^*f\geq0$ $\m$-a.e. and $\|T_t^*f\|_{L^1(E;\,\m)}\leq\|f\|_{L^1(E;\,\m)}$. 
\end{itemize}

Under $(\mathscr{E}\!.\,1)$, $(\mathscr{E}\!.\,2)$, $(\mathscr{E}\!.\,3)$ and $(\mathscr{E}\!.\,4)$, 
$\{T_t\}_{t\geq0}$ and $\{T_t^*\}_{t\geq0}$  are positivity preserving 
in the sense that 
for $f\in L^2(E;\,\m)$, $f\geq 0$ $\m$-a.e. implies  $T_tf\geq0$ $\m$-a.e. and  $T_t^*f\geq0$ $\m$-a.e. (see \cite[Remark~1.4~(i), (iii) and Theorem~1.5]{MRPosiPre}).  Under $(\mathscr{E}\!.\,4)$, 
$\{T_t\}_{t\geq0}$ (resp.~$\{G_{\alpha}\}_{\alpha>0}$) can be extended to a bounded linear operator on $L^{\infty}(E;\,\m)$ for $t>0$ (resp.~$\alpha>0$) (see \cite[p.~8]{Oshima}). 
As shown in \cite[p.~20]{Oshima}, for $t>0$, $\alpha>0$, and $f \in L^0_+(E;\,\m)$, we define 
$T_tf$ and $G_{\alpha}f$ (resp.~$T_t^*f$ and $G_{\alpha}^*f$) by
$T_tf=\lim_{n\to\infty}T_t(f\land nh)$ and $G_{\alpha}f=\lim_{n\to\infty}G_{\alpha}(f\land nh)$ (resp.~$T_t^*f=\lim_{n\to\infty}T_t^*(f\land nh)$ and $G_{\alpha}^*f=\lim_{n\to\infty}G_{\alpha}^*(f\land nh)$). Here, 
$L^0_+(E;\,\m)$ denotes the family of all non-negative $\m$-measurable functions and 
$h\in L^{\infty}(E;\,\m)\cap L^1(E;\,\m)$ is a strictly positive function. Then, we have the following generalized duality relation: for non-negative measurable functions $f,g$,  
\begin{equation*}
\int_ET_tf\,g\, \d\m=\int_Ef\,T_t^*g\,\d\m,\quad \int_EG_{\alpha}f\,g\, \d\m=\int_Ef\,G_{\alpha}^*g\,\d\m
\end{equation*}
for $t>0$ and $\alpha>0$. 

The lower bounded semi-Dirichlet form $(\mathscr{E},\mathscr
{F})$ on $L^2(E;\,\m)$ is said to be \emph{regular} if $\mathscr{F}\cap C_0(E)$ is $\mathscr{E}_{\alpha_0+1}^{1/2}$-dense in $\mathscr{F}$ and $\mathscr{F}\cap C_0(E)$ is 
uniformly dense in $C_0(E)$. Under the regularity of $(\mathscr{E},\mathscr{F})$, 
there exists a Hunt process ${\bf X}$ associated with $(\mathscr{E},\mathscr{F})$ in the sense that for $u\in L^{\infty}(E;\,\m)\cap \mathscr{B}(E)$, $G_{\alpha}u=R_{\alpha}u$ $\m$-a.e.~for each $\alpha>0$ (see \cite[\S3.3, Theorem~3.3.4]{Oshima}). 
 Moreover, we can prove that the semigroup $P_tu$ with $u\in L^2(E;\,\m)\cap\mathscr{B}(E)$ is a quasi-continuous $\m$-version of 
 $T_tu$ for $t>0$ by the same way of the proof of \cite[Chapter IV, Proposition~2.8]{MR} with the help of \cite[\S2.2, Theorem~2.2.5]{Oshima}.  
Then, we obtain the next proposition.
\begin{prop}\label{prop:L2semigroup}
\begin{enumerate}
\item Suppose that ${\bf (A1)}$ is satisfied. For any $p\in[1,+\infty)$, 
the semigroup $\{P_t\}_{t\geq0}$ of ${\bf X}$ is extended to a strongly continuous contraction semigroup $\{T_t\}_{t\geq0}$ on $L^p(E;\,\m)$. 
\item Suppose that ${\bf (A2)}$ is satisfied  
and let $-\alpha_0$ be the lower bound 
of $(\mathscr{E},\mathscr{F})$.  
Then, for any $p\in [2,+\infty)$, the semigroup 
$\{ e^{-\alpha_0(2/p)t}P_t\}_{t \ge 0}$ of ${\bf X}$ is extended to a strongly continuous contraction semigroup 
$\{ e^{-\alpha_0(2/p)t}T_t\}_{t \ge 0}$ on $L^p(E;\,\m)$. In particular, $\{T_t\}_{t\geq0}$ is strongly continuous on $L^p(E;\,\m)$. 
\end{enumerate}
\end{prop}

\begin{proof}
We first prove a). Suppose that {\bf (A1)} is satisfied.
Jensen's inequality and \eqref{eq:exc}  imply that for $f \in C_0 (E)$ and $t \in [0,+\infty )$
\begin{align}
&\| P_t f \| _{L^p(E;\,\m)}^p= \int _E | \bE_x[ f(X_t)] |^p \,\m(\d x) \le \int _E \bE_x[ |f(X_t)|^p ] \,\m(\d x) \notag \\
&\hspace{0.5cm}\le \int _E P_t (|f|^p) (x) \,\m (\d x) \le\int _E |f(x)|^p \,\m (\d x) \le \| f\|_{L^p(E;\,\m)}^p .\label{eq:contr}
\end{align}
Note that  $C_0 (E)$ is dense in $L^p(E;\,\m)$. Then, \eqref{eq:contr} implies that $\{ P_t\}_{t \ge 0}$ is extended to a contraction semigroup on $L^p(E;\,\m)$, which is denoted by $\{T_t\}_{t \ge 0}$.

Let $f \in L^p(E;\,\m)$ given. Then, for any $\varepsilon >0$, there exists $g\in C_0 (E)$ such that $
\| f-g\| _{L^p(E;\,\m)}< \varepsilon/2.$ By using \eqref{eq:contr}, we obtain that for any $t>0$,
\begin{align}
\| T_t f -f \| _{L^p(E;\,\m)}&\le  \| T_t f - T_t g \| _{L^p(E;\,\m)} + \| T_t g -g \| _{L^p(E;\,\m)} + \| g-f \| _{L^p(E;\,\m)} \notag \\
&\le \varepsilon+\| T_t g -g \| _{L^p(E;\,\m)}. \label{eq:s}
\end{align}
 It is easy to see that 
 for any $f\in L^p(E;\,\m) \cap \mathscr{B}_b(E)$, $T_t f(x) = P_t f(x)$ for $\m$-a.e. $x \in E$.
Then, by using the sample path (right-)continuity of ${\bf X}$, we have
\[
\varlimsup _{t \downarrow 0} \| T_t g -g \| _{L^p(E;\,\m)}=\varlimsup _{t \downarrow 0} \| P_t g -g \| _{L^p(E;\,\m)}=0.
\]
Therefore, \eqref{eq:s} implies  $\varlimsup _{t \downarrow 0}\| T_t f -f \| _{L^p(E;\,\m)}\le \varepsilon.$ Since $\varepsilon>0$ is arbitrarily chosen, we complete the proof.

Next we prove b). 
Suppose {\bf (A2)} is satisfied. Since $\{T_t^*\}_{t\geq0}$ is $L^1(E;\,\m)$-contractive and 
$\{e^{-\alpha_0t}T_t^*\}_{t\geq0}$ is $L^2(E;\,\m)$-contractive, 
we get $\{e^{-\alpha_0(2(q-1)/q)t}T_t^*\}_{t\geq0}$ is 
$L^q(E;\,\m)$-contractive for $q\in[1,2]$ in view of the Riesz-Thorin interpolation 
theorem (see \cite[1.1.5]{Da0}).  
Take a relatively compact open  set $G$. For $f\in C_0(E)$ and $p\in[2,+\infty)$, we see $P_tf\in L^p(G;\,\m)$:
\begin{align}
\|P_tf\|_{L^p(G;\,\m)}^p&\leq\int_GP_t|f|^p\,\d\m=\int_ET_t^*\bone_G(x)|f(x)|^p\,\m(\d x)\notag\\
&\leq\|f\|_{\infty}^p\int_ET_t^*\bone_G(x)\,\m(\d x)=\|f\|_{\infty}^p\int_ET_t\bone_E(x)\bone_G(x)\,\m(\d x)\label{eq:contr*}
\\&
\leq
\|f\|_{\infty}^p\m(G)<\infty.\notag
\end{align}
Let $q=p/(p-1) \in (1,2)$.  Since $\{e^{-\alpha_0(2(q-1)/q)t}T_t^*\}_{t \ge 0}$ is $L^q(E;\,\m)$-contractive and  $\{T_t^\ast\}_{t \ge 0}$ is positivity preserving, we obtain that for any $t>0$,
\begin{align}
\int_G|P_tf(x)|^p\,\m(\d x)&=\sup_{g\in L^q(G;\,\m),\|g\|_{L^q(G;\,\m)}=1}\left|\int_GP_tf(x)g(x)\,\m(\d x) \right|\notag \\
&=\sup_{g\in L^q(G;\,\m),\|g\|_{L^q(G;\,\m)}=1}\left|\int_Ef(x)T_t^*(\bone_Gg)(x)\,\m(\d x) \right|\notag \\
&\leq\sup_{g\in L^q(E;\,\m),\|g\|_{L^q(E;\,\m)}=1}\left|\int_Ef(x)T_t^*g(x)\,\m(\d x) \right|\notag\\
&\leq\sup_{g\in L^q(E;\,\m),\|g\|_{L^q(E;\,\m)}=1} \|f\|_{L^p(E;\,\m)}\|T_t^*g\|_{L^q(E;\,\m)}\notag\\
&\leq e^{\alpha_0(2(q-1)/q)t}
\sup_{g\in L^q(E;\,\m),\|g\|_{L^q(E;\,\m)}=1} \|f\|_{L^p(E;\,\m)}\|g\|_{L^q(E;\,\m)}\notag\\
&=e^{\alpha_0(2/p)t}
\|f\|_{L^p(E;\,\m)}.
\notag
\end{align}
Since $G$ is arbitrary, 
$\{ e^{-\alpha_0(2/p)t}P_t\}_{t \ge 0}$ is extended to a contraction semigroup on $L^p(E;\,\m)$, which is denoted by $\{e^{-\alpha_0(2/p)t}T_t\}_{t \ge 0}$.
Let $f \in L^p(E;\,\m)$ given. Then, for any $\varepsilon >0$, there exists $g\in C_0 (E)$ such that $
\| f-g\| _{L^p(E;\,\m)}< \varepsilon/2.$ By using \eqref{eq:contr*}, we obtain that for any $t>0$,
\begin{align*}
&\| T_t f -f \| _{L^p(E;\,\m)} \notag \\
&\hspace{0.5cm}\le  \| T_t f - T_t g \| _{L^p(E;\,\m)} + \| T_t g -g \| _{L^p(E;\,\m)} + \| g-f \| _{L^p(E;\,\m)}  
\\
&\hspace{0.5cm}\le \varepsilon{(1+e^{\alpha_0(2/p)t})}+\| T_t g -g \| _{L^p(E;\,\m)}.\notag
\end{align*}
 The rest of the proof is similar to that of a).
\end{proof}

\begin{rem}\label{rem:rightcontinuity}
{\rm 
In the proof of Proposition~\ref{prop:L2semigroup}, the sample path right-continuity of ${\bf X}$ is used. However, the right-continuity of $\{P_t\}_{t \ge 0}$ on $C_0(E)$ and the fact that $C_0(E)$ is a dense subset of $L^p(E;\m)$ play an  essential role.}
\end{rem}

In the sequel, we fix $p \in (1,+\infty)$ under ${\bf (A1)}$ with $\alpha_0:=0$ (resp.~$p\in[2,+\infty)$ under ${\bf (A2)}$), and let $\{ T_t\}_{t \ge 0}$ be the  strongly continuous contraction semigroup as in Proposition~\ref{prop:L2semigroup}. For $f \in L^p(E;\,\m)$ and $\alpha \in (2\alpha_0/p,+\infty)$, we put
\[
G_{\alpha }f=\int_{0}^\infty e^{-\alpha t}T_{t}f\,\d t,
\]
the integral being defined as the Bochner integral in $L^p(E;\,\m)$. Then, $\{G_{\alpha}\}_{\alpha >2\alpha_0/p}$ becomes a strongly continuous contraction resolvent on $L^p(E;\,\m)$ in the following sense that 
$\lim_{\alpha\to\infty}\alpha G_{\alpha}f=f$ for $f\in L^p(E;\,\m)$; 
$(\alpha -2\alpha_0/p)\|G_{\alpha}f\|_{L^p(E;\,\m)}\leq\|f\|_{L^p(E;\,\m)}$ for $f\in L^p(E;\,\m)$ and 
$\alpha>2\alpha_0/p$
; $G_{\alpha}-G_{\beta}+(\alpha-\beta)G_{\alpha}G_{\beta}=0$ for 
all $\alpha,\beta>2\alpha_0/p$. 
 Moreover, we have $R_{\alpha }f=G_{\alpha }f$ $\m$-a.e. on $E$  for any $f \in  L^p(E;\,\m)\cap \mathscr{B}_b(E)$ and $\alpha \in (2\alpha_0/p,+\infty)$. The generator $(A,\text{Dom}(A))$ of $\{T_t\}_{t \ge 0}$ is defined by
\begin{align*}
\text{Dom}(A)&:=\left\{f \in L^p(E;\,\m) \;\left|\; \lim_{t \to 0} (T_t f-f)/t \text{ exists in }L^p(E;\,\m)\right.\right\}, \\
Af&:=\lim_{t \to 0}\frac{T_t f-f}{t},\quad f \in \text{Dom}(A).
\end{align*}
It is known that $\text{Dom}(A)$ is equal to $G_\alpha(L^p(E;\,\m))$ for some/any $\alpha \in (2\alpha_0/p,+\infty)$. See \cite[Chapter~I, Proposition~1.5]{MR} for details.

Next, we define the transition kernel of ${\bf X}$. For $t>0$, $x \in E$, and $B \in \mathscr{B}(E)$, we set
\[
P_t(x,B):=\bE_x[\bone_{B}(X_t)]=\bP_x(X_t \in B).
\]
The resolvent kernel of ${\bf X}$ is defined as \[R_\alpha(x,B)=\int_{0}^\infty e^{-\alpha t}P_t(x,B)\,\d t,\quad \alpha>0,\; x \in E\; \text{, and }\;B\in \mathscr{B}(E).\]   The transition kernel (resp. the resolvent kernel) is often said  
to be \emph{absolutely continuous with respect to $\m$} if  $P_t(x,\cdot)$ (resp. $R_\alpha(x,\cdot)$) is absolutely continuous with respect to $\m$ for any $x \in E$ and $t>0$ (resp. $\alpha>0$).

\begin{lem}\label{lem:density}
\begin{enumerate}
\item \label{lem:density1} If $\{P_t\}_{t \ge 0}$ has the strong Feller property, then $P_t(x,\cdot )$ is absolutely continuous with respect to $\m$ for all $x\in E$ and $t > 0$.
\item \label{lem:density2} If $\{R_{\alpha}\}_{\alpha >0}$ has the strong Feller property, then  $R_\alpha (x,\cdot )$ is absolutely continuous with respect to $\m$ for all $x\in E$ and $\alpha>0$.
\end{enumerate}
\end{lem}

\begin{proof}
We prove a). Fix $t>0$ and $B \in \mathscr{B}(E)$ with $\m(B)=0$.
Then, $\bone_{B} \in L^2(E;\,\m)$ and for any $g\in L^2(E;\,\m)$, \[
\int _{E} g(x) T_{t} \bone_{B}(x) \,\m(\d x) = \int _{B} T_{t}^* g(x) \,\m(\d x) =0.\]
This implies that $T_{t}\bone_{B}(x) =0$ for $\m$-a.e. $x \in E$, and hence $P_{t}\bone_{B}(x) =0$ for $\m$-a.e. $x \in E$. On the other hand, by the strong Feller property of $\{P_t\}_{t \ge 0}$, $P_{t}\bone_{B}$ is continuous on $E$. 
Thus, $P_{t}\bone_{B}(x) =0$ for any $x\in E$, and it yields the absolute continuity. The proof of \ref{lem:density2} is almost same. So, we omit it.
\end{proof}

In what follows, if no confusion will arise, the complexifications of subspaces of any real Banach space  and linear operators on them are denoted by the same symbols. 
So any functions in $L^p(E;\,\m)$ are regarded as complex valued if there is no special remark. 
With this notation, we give the definitions of analytic semigroups and sectorial operators. 
 For $\theta \in (0,\pi)$, we define \[S_\theta:=\{ \lambda \in {\mathbb C} \setminus \{ 0\} \mid |\arg \lambda | < \theta\}.\] 

\begin{defi}[{\cite[Definition~2.5.1]{Pa}}]\label{def:analytic}
\rm The semigroup $\{T_t\}_{t\ge 0}$ on $L^p(E;\,\m)$ is said to be \emph{analytic} if there exists $\delta \in (0,\pi/2]$ such that $\{ T_t\}_{t \ge 0}$ is extended to a family of operators $\{ T_z\}_{z\in S_{\delta}\cup \{ 0\}}$ on $L^p(E;\,\m)$ 
with the following properties:
\begin{enumerate}
\item $T_{z_1 + z_2} = T_{z_1} T_{z_2}$ for $z_1, z_2 \in S_{\delta}$,
\item $z \mapsto T_z$ is analytic on $S_{\delta}$,
\item $\lim _{z\in S_{\varepsilon},\, z\to 0} T_z f = f$ for any $f\in L^p(E;\,\m)$ and $\varepsilon \in (0,\delta)$.
\end{enumerate}
\end{defi}

The resolvent set and the resolvent operator of the generator $(A,\text{Dom}(A))$ are  denoted by $\rho(A)$ and $\{R(z\,;A)\}_{z \in \rho(A)}$, respectively. From the definitions, it is obvious that $ \rho(A)$ includes the positive half line $(0,+\infty)$.
\begin{defi}[{\cite[(5.1), (5.2)]{Pa}, \cite[Definition~2.0.1]{Lu}}]\label{def:sectorial}
{\rm \quad The generator \\ $(A,\text{\rm Dom}(A))$ on $L^p(E;\,\m)$ is said to be  \emph{sectorial} if there exist constants $\theta \in (\pi /2, \pi)$ and $M \in (0,+\infty)$ such that $\rho (A) \supset S_{\theta}$ and 
\begin{align}\label{eq:sec}
 \quad \|z R(z\,; A) f\|_{L^p(E;\,\m)} \le M \|f\|_{L^p(E;\,\m)},\quad z \in S_{\theta},\,\, f \in L^p(E;\,\m).
\end{align}
}
\end{defi}
We often say that $(A,\text{\rm Dom}(A))$ is a \emph{sectorial operator on $L^p(E;\,\m)$ with constants $\theta \in (\pi /2, \pi)$ and $M \in (0,+\infty)$} if it satisfies \eqref{eq:sec}. 

\vspace{0.3cm}

The next proposition is a special case of \cite[Theorems~1.7.7 and 2.5.2]{Pa}

\begin{prop}\label{prop:anal}
Let $\{ T_t\}_{t \ge 0}$ be the strongly continuous 
semigroup on $L^p(E;\,\m)$ constructed in Proposition~\ref{prop:L2semigroup}. Then, 
$\{ T_t\}_{t \ge 0}$ is analytic if and only if its generator $A$ is sectorial.  
In this case, $\{T_t\}_{t \ge 0}$ is expressed by the Dunford integral as
\begin{align*}
T_t = \frac{1}{2\pi {\rm i}} \int _{\Gamma _{\theta, \varepsilon}} e^{z t} R(z\,;A)\, \d z, \quad t\in (0,+\infty ),
\end{align*}
where $\theta$ is the constant as in Definition~\ref{def:sectorial}, and
\begin{align*}
\Gamma _{\theta, \varepsilon} &:= \{ \lambda \in {\mathbb C}\setminus \{ 0 \}\,\mid\, |\lambda |\geq \varepsilon , |\arg \lambda | = \theta -\varepsilon \} \\
&\quad \cup \{ \lambda \in {\mathbb C}\setminus \{ 0\}\,\mid\, |\lambda | = \varepsilon , |\arg \lambda | \leq \theta -\varepsilon\}, \quad \varepsilon \in (0, \theta -\pi /2).
\end{align*}
\end{prop}

\begin{thm}\label{thm:anal}
Let $\{ T_t\}_{t \ge 0}$ be the  strongly continuous  semigroup  constructed in Proposition~\ref{prop:L2semigroup}. 
Then we have the following: 
\begin{enumerate}
\item Suppose that ${\bf (A1)}$ holds and $\{ T_t\}_{t \ge 0}$ is analytic on $L^2(E;\,\m)$ with the sector $S_{\delta}$ for some $\delta \in (0,\pi/2]$. Then 
$\{ T_t\}_{t \ge 0}$ is analytic on $L^p(E;\,\m)$ for $p\in(1,+\infty)$ with the sector $S_{\delta'}$, where $\delta':=\delta\left(1-\left|(2/p)-1 \right|\right)$. 
\item Suppose that ${\bf (A2)}$ holds. Then $\{ T_t\}_{t \ge 0}$ is analytic on $L^p(E;\,\m)$ for $p\in[2,+\infty)$ with the sector $S_{\delta'}$, where $\delta':=(2/p)\arctan K^{-1}$ and $K\geq1$ is the constant specified in 
$(\mathscr{E}\!.\,2)$.  
\end{enumerate}
\end{thm}
\begin{proof}
The proof can be done similarly as in the proof of \cite[Chapter III, Theorem~1]{Stein} by using the convexity theorem in \cite[p.~69]{Stein}. 
Combining the same proof of \cite[Chapter~I, Corollary~2.21]{MR} with \eqref{eq:WSC},  ${\bf (A2)}$  
implies that $\{ e^{-\alpha t}T_t\}_{t \ge 0}$ with $\alpha>\alpha_0$ is analytic on $L^2(E;\,\m)$ with sector $S_{\arctan K^{-1}}$.  
\end{proof}
\begin{rem}\label{rem:analytic}
\rm 
\begin{enumerate}
\item[(i)] 
Let $\{T_t\}_{t\geq0}$ be a strongly continuous symmetric  
semigroup on $L^2(E;\,\m)$. 
Then $\{ T_t\}_{t \ge 0}$ is analytic on $L^2(E;\,\m)$ with the sector 
$S_{\pi/2}$ as proved in \cite[Chapter~III, Theorem~1]{Stein}. 
\item[(ii)] Let $\{T_t\}_{t\geq0}$ be a strongly continuous contraction 
semigroup on $L^2(E;\,\m)$ associated with a coercive closed form 
$(\mathscr{E},\mathscr{F})$ on $L^2(E;\,\m)$ in the sense of \cite[Chapter~I, Definition~2.4]{MR}. Then $\{e^{-t}T_t\}_{t\geq0}$ (equivalently, $\{T_t\}_{t\geq0}$)  is analytic on $L^2(E;\,\m)$ 
with the sector $S_{\delta}$ for $\delta:=\arctan K^{-1}$, where $K\,\geq1$ is the constant 
appeared in the weak sector condition for $(\mathscr{E},\mathscr{F})$ 
(see \cite[Chapter I, (2.3) and Corollary~2.21]{MR}, where $K>0$ is only noted, but indeed, $K\geq1$ holds automatically).  If further $(\mathscr{E},\mathscr{F})$ satisfies the strong sector condition with $K \geq 1$ (see  \cite[Chapter~I, (2.4)]{MR} for strong sector condition), then 
$\{T_t\}_{t\geq0}$ is analytic on $L^2(E;\,\m)$ 
with the sector $S_{\delta}$ having the same expression as above. This assertion is not so sharp when 
$\{T_t\}_{t\geq0}$ is $\m$-symmetric, in this case $K=1$ so that $\delta=\pi/4$. 
\item[(iii)] Let $\{T_t\}_{t\geq0}$ be a strongly continuous contraction 
semigroup on $L^2(E;\,\m)$ associated with a (non-negative definite) 
non-symmetric Dirichlet form 
$(\mathscr{E},\mathscr{F})$ on $L^2(E;\,\m)$ (see \cite[Chapter I, Definition~4.5]{MR} for non-symmetric Dirichlet form).  Then, for $p\in(1,+\infty)$, 
one can prove that $\{ e^{-t}T_t\}_{t \ge 0}$ (equivalently, $\{T_t\}_{t\geq0}$) is analytic on $L^p(E;\,\m)$ with the sector $S_{\delta'}$ in the same way of the proof of Theorem~\ref{thm:anal}, where $\delta':=(\arctan K^{-1})\left(1-\left| (2/p)-1 \right|\right)$ for some $K \ge 1$ 
derived from the weak sector condition for $(\mathscr{E},\mathscr{F})$.  
If further $(\mathscr{E},\mathscr{F})$ satisfies the strong sector condition with $K \ge 1$, then $\{T_t\}_{t\geq0}$ is analytic on $L^p(E;\,\m)$ 
with the sector $S_{\delta'}$ having the same expression as above. 
\end{enumerate}
\end{rem}

\section{Equivalence of the strong Feller properties}\label{sec:equi}

In this section, we use the same notation as in Section~\ref{sec:stoc}.
We only consider a Hunt process ${\bf X}$ and assume that the semigroup $\{P_t\}_{t\geq0}$ 
can be extended to a strongly continuous semigroup $\{T_t\}_{t\geq0}$ on $L^p(E;\,\m)$ for some $p\in[1,+\infty)$. We fix such a $p\in[1,+\infty)$.  

\begin{prop}\label{prop:rstFeller}
Fix $p \in [1,+\infty)$ as above and assume the following conditions are satisfied:
\begin{itemize}
\item[{\rm (i)}] $(A,\text{\rm Dom}(A))$ is a sectorial operator on $L^p(E;\,\m)$ with  constants $\theta \in (\pi /2, \pi)$ and $M \in (0,+\infty)$.
\item[{\rm (ii)}] There exists $\varepsilon \in (0, \theta -\pi /2)$ such that for any $z \in \Gamma _{\theta, \varepsilon}$ and $f\in  C_b(E) \cap L^p(E;\,\m)$, $R(z\,;A)f$ possesses a bounded continuous $\m$-version on $E$.
\item[{\rm (iii)}] $\{ R_\alpha\}_{\alpha>0}$ has the strong Feller property. 
\end{itemize}
Then, for any $f\in  L^p(E;\,\m)\cap {\mathscr B}_b(E)$, there exist a ${\mathbb C}$-valued Borel measurable function $\overline{Rf}$ on $\Gamma _{\theta, \varepsilon} \times E$ and an $\m$-null set $N \subset E$ with the following properties:
\begin{enumerate}
\item\label{prop:rstFeller1} $\overline{Rf}(z,x) = R(z\,;A)f(x)$ for $|\d z|\otimes \m$-almost every $(z,x) \in \Gamma_{\theta, \varepsilon}\times E$,

\item \label{prop:rstFeller3} $\overline{Rf}(\cdot ,x)$ has a $\{(p-1)/p\}$-H\"older continuous version for any $x\in E$,

\item \label{prop:rstFeller4} for any compact subset $K \subset E$,
\begin{align*}
\lim_{j \to \infty}\sup_{x,y \in K \setminus N;\ d(x,y)<1/j} \left|\overline{Rf}(\cdot,x)-\overline{Rf}(\cdot,y) \right|=0,\quad |\d z|\text{\rm -a.e. on }\Gamma_{\theta,\varepsilon}.
\end{align*}
\end{enumerate}
\end{prop}

\begin{proof}
Fix $f\in L^p(E;\,\m) \cap {\mathscr B}_b(E)$.
In view of the definitions of $\{ R_\alpha\}_{\alpha>0}$ and $\{R(z\,;A)\}_{z \in \rho(A) }$,
\begin{equation}\label{eq:rstFeller}
 R(1\,;A) f(x)=R_{1}f (x), \quad \m\text{-a.e. }x \in E.
\end{equation}
The resolvent equation implies that
\begin{equation}\label{eq:proprstFeller00}
R(z\,;A) = R(1\,;A)+ (1-z)R(z\,;A)R(1\,;A), \quad z\in \rho(A),
\end{equation}
in the sense of bounded linear operators on $L^p(E;\,\m)$.
Hence, by \eqref{eq:rstFeller} and \eqref{eq:proprstFeller00}, we have for any $z\in \rho(A)$, 
\begin{equation}\label{eq:proprstFeller01}
R(z\,;A)f(x) = R_{1} f(x)+ (1-z)R(z\,;A)R_{1} f(x), \quad \m\text{-a.e.}\ x \in E.
\end{equation}
Then, the assumption and (\ref{eq:proprstFeller01}) imply that for any $z\in \Gamma _{\theta ,\varepsilon}$, there exist ${V}_zf \in C_b(E)$ and an $\m$-null set $N_1(z)$ (depending on $z$) such that 
\begin{equation}\label{eq:proprstFeller01-2}
V_zf(x) = R(z\,;A)f(x), \quad x \in E \setminus N_1(z).
\end{equation}
By \cite[Theorem~1.3.2]{DuPe}, the mappings $(z,x) \mapsto R(z\,;A)f(x)$ and $(z,x) \mapsto V_zf(x)$ possess Borel measurable versions on $\Gamma _{\theta ,\varepsilon} \times E$. Then, we see from  \eqref{eq:proprstFeller01-2} that
\begin{align}
\iint _{\Gamma _{\theta ,\varepsilon}\times E} |V_zf(x) - R(z\,;A)f(x) |^p \,|\d z|\otimes \m(\d x) =0.\label{eq:Fubini1} 
\end{align}
Hence, by Fubini's theorem, there exists an $\m$-null set $N_1 \subset E$ such that for any $x\in E\setminus N_1$
\begin{equation}\label{eq:proprstFeller01-3}
V_zf(x) = R(z\,;A)f(x), \quad |\d z|\text{-a.e.}\ z \in \Gamma _{\theta ,\varepsilon}.
\end{equation}

On the other hand, we see from \cite[(5.21)]{Pa} that $
(\d/\d z)R(z\,;A) = -R(z\,;A)^2$,  $z\in \rho(A)$,
in the sense of bounded linear operators on $L^p(E;\,\m)$.
Therefore, we obtain that for any $z\in S_\theta$,
\begin{align*}
\left\| \frac{\d}{\d z}R(z\,;A) f\right\|_{L^p(E;\,\m)} & \le \| R(z\,;A)\| ^2_{L^p \to L^p} \| f\| _{L^p(E;\,\m)}\le M^2 |z|^{-2} \| f\| _{L^p(E;\,\m)}.
\end{align*}
Hence, we obtain that
\begin{align*}
\int _E \left( \int _{\Gamma _{\theta ,\varepsilon}}\left| \frac{\d}{\d z} R(z\,;A)f (x) \right| ^p |\d z| \right) \m(\d x)
&= \int _{\Gamma _{\theta ,\varepsilon}}\left\| \frac{\d}{\d z} R(z\,;A)f \right\| _{L^p(E;\,\m)}^p |\d z|\\
&\le M^{2p} \| f\| _{L^p(E;\,\m)}^p  \int _{\Gamma _{\theta ,\varepsilon}} |z|^{-2p} |\d z|< \infty .
\end{align*}
This estimate implies
\[
R(\,\cdot\,;A) f (x) \in W^{1,p}(\Gamma _{\theta ,\varepsilon}; {\mathbb C}), \quad \m\text{-a.e. }x \in E,
\]
where we denote by $W^{1,p}(\Gamma _{\theta ,\varepsilon}; {\mathbb C})$ the $\C$-valued $p$-th order Sobolev space on $\Gamma _{\theta ,\varepsilon}$. Then, by using the Sobolev embedding theorem and the Fubini's theorem, we obtain an $\m$-null set $N_2 \subset E$ with the following property:  for any $x \in E\setminus N_2$, there exists $W_{{}_{\bullet}}f(x) \in C^{(p-1)/p}(\Gamma _{\theta ,\varepsilon}; {\mathbb C})$ such that 
\begin{equation}\label{eq:proprstFeller02}
R(z \,;A) f (x) =W_zf(x), \quad |\d z|\mbox{-almost every}\ z \in \Gamma _{\theta ,\varepsilon}.
\end{equation}
Here, we denote by $C^{(p-1)/p}(\Gamma _{\theta ,\varepsilon}; {\mathbb C})$ the space of ${\mathbb C}$-valued $\{(p-1)/p\}$-H\"older continuous functions on $\Gamma _{\theta ,\varepsilon}$.

Finally, we let $N:=N_1 \cup N_2$, and define a ${\mathbb C}$-valued function $\overline{Rf}$ on $\Gamma _{\theta ,\varepsilon}\times E$ by
\begin{align*}
\overline{Rf}(z ,x):= 
\begin{cases}
V_zf (x), & (z,x)\in \Gamma _{\theta ,\varepsilon} \times (E\setminus N),\\
0,& (z,x)\in \Gamma _{\theta ,\varepsilon} \times N.
\end{cases}
\end{align*}
Then, properties~\ref{prop:rstFeller1} and \ref{prop:rstFeller3} follow from \eqref{eq:Fubini1}, \eqref{eq:proprstFeller01-3} and \eqref{eq:proprstFeller02}. The property~\ref{prop:rstFeller4} is a consequence of  the continuity of $V_zf(x)$ in $x \in E$.
\end{proof}

We do not know whether the following local uniform estimates can be obtained only from the conditions in Proposition~\ref{prop:rstFeller}: for any $t\in (0,+\infty)$, $f \in  L^p(E;\,\m)\cap \mathscr{B}_b(E)$, and compact subsets $K \subset E$ and $L \subset \mathbb{C}$,
\begin{align}
&\left\|\int _{\Gamma _{\theta, \varepsilon}} e^{t\,\text{Re\,z}} \left|\overline{Rf} (z,\cdot)\right| \,|\d z| \right\|_{L^{\infty}(K;\,\m)}  <\infty, \label{eq:unkn1}\\
&\left\|\overline{Rf}(\cdot,\cdot)\right\|_{L^\infty((L \cap \Gamma_{\theta,\varepsilon})\times K;\, |\d z|\otimes \m)}<\infty \label{eq:unkn2}.
\end{align}

To confirm \eqref{eq:unkn2}, we provide the following criterion as a proposition. 

\begin{prop}\label{prop:rstFeller2}
Fix $p \in [1,+\infty)$ as above, and assume that $(A,\text{\rm Dom}(A))$ is a sectorial operator on $L^p(E;\,\m)$ with  constants $\theta \in (\pi /2, \pi)$ and $M \in (0,+\infty)$. Let  $\varepsilon \in (0,\theta -\pi/2)$, and assume in addition that  $L^p(E;\,\m) \cap L^\infty(E;\m)$ is invariant under $R(z;A)$  
 for any  $z \in \Gamma_{\theta,\varepsilon}$. Then, for any $f \in L^p(E;\m) \cap L^\infty(E;\m),$ the map $\Gamma_{\theta,\varepsilon} \ni z \mapsto \|R(z;A)f\|_{L^\infty(E;\,\m)}$ is continuous. 
\end{prop}
\begin{proof}
For $f \in L^p(E;\,\m) \cap L^{\infty}(E;\,\m)$, we set $\|f\|_{p,\infty}=\|f\|_{L^p(E;\,\m)}\lor \|f\|_{L^{\infty}(E;\,\m)}$. Then, $L^p(E;\,\m) \cap L^{\infty}(E;\,\m)$ is a Banach space under $\|\cdot\|_{p,\infty}$.  We fix $\varepsilon \in (0,\theta -\pi/2)$ and $z \in \Gamma_{\theta,\varepsilon}$. By noting that $R(z;A)$ is a bounded linear operator on $L^p(E;\,\m)$ and makes invariant $L^p(E;\,\m) \cap L^\infty(E;\,\m)$, we immediately find that $R(z;A)$ is a closed operator on $L^p(E;\,\m) \cap L^{\infty}(E;\,\m)$. Hence,  the closed graph  theorem implies that there exists $C_z>0$ such that  for any $f \in L^p(E;\,\m) \cap L^{\infty}(E;\,\m)$,
\[
\|R(z;A)f\|_{p,\infty} \le C_z\|f\|_{p,\infty}.
\]
In particular, we have that for $f \in L^p(E;\,\m) \cap L^{\infty}(E;\m)$,
\begin{align}\label{eq:brz}
\|R(z;A)f\|_{p,\infty} \le C_z(\|f\|_{L^p(E;\,\m)}+\|f\|_{L^{\infty}(E;\,\m)}).
\end{align}
We fix $f \in L^p(E;\,\m) \cap L^{\infty}(E;\m)$ and $z_0 \in \Gamma_{\theta,\varepsilon}$.  From the resolvent equation and \eqref{eq:brz}, we obtain that for any $z \in \Gamma_{\theta,\varepsilon}$,
\begin{align}
&\|R(z;A)f-R(z_0;A)f\|_{L^\infty(E;\,\m)} \le |z-z_0|\|R(z_0;A)R(z;A)f\|_{L^\infty(E;\,\m)} \notag \\
&\hspace{2cm}\le C_{z_0}|z-z_0|(\|R(z;A)f\|_{L^p(E;\,\m)}+\|R(z;A)f\|_{L^\infty(E;\,\m)}) 
\label{eq:brz2} \\
&\hspace{2cm}\le C_{z_0}\frac{M|z-z_0|}{|z|}\|f\|_{L^p(E;\,\m)}+C_{z_0}|z-z_0|\|R(z;A)f\|_{L^\infty(E;\,\m)}.\notag
\end{align}
In the last line, we use the fact that $A$ is a sectorial operator.
Therefore, for any $z \in  \Gamma_{\theta,\varepsilon}$ with $C_{z_0}|z-z_0|<1/2$, 
\begin{equation}
\frac{1}{2}\left\| R(z;A)f\right\| _{L^\infty(E;\,\m)} \le \|R(z_0;A)f\|_{L^\infty(E;\,\m)}+C_{z_0}\frac{M|z-z_0|}{|z|}\|f\|_{L^p(E;\,\m)}.\label{eq:brz3}
\end{equation}
By using \eqref{eq:brz2} and \eqref{eq:brz3}, we have for any $z \in  \Gamma_{\theta,\varepsilon} \text{ with }C_{z_0}|z-z_0|<1/2$,
\begin{align*}
&\|R(z;A)f-R(z_0;A)f\|_{L^{\infty}(E;\,\m)} \\
&\hspace{1cm}\le C_{z_0}\frac{M|z-z_0|}{|z|}\|f\|_{L^p(E;\,\m)}+2C_{z_0}|z-z_0||R(z_0;A)f\|_{L^\infty(E;\,\m)}\\
&\hspace{2cm} +2C_{z_0}^2\frac{M|z-z_0|^2}{|z|}\|f\|_{L^p(E;\,\m)}.
\end{align*}
This shows that the map $z \mapsto \|R(z;A)\|_{L^{\infty}(E;\,\m)}$ is continuous at $z_0$.
\end{proof}

\begin{thm}\label{thm:equi}
Fix $p \in [1,+\infty)$ as above, and assume the conditions in Proposition~\ref{prop:rstFeller},  \eqref{eq:unkn1} and \eqref{eq:unkn2}. Assume in addition that  the transition kernel of ${\bf X}$ is absolutely continuous with respect to $\m$.
Then, for any $t>0$ and  $f \in  L^p(E;\m)\cap \mathscr{B}_b(E)$, we have  $P_t f \in C_b(E)$. In particular, $\{P_t\}_{t \ge 0}$ has the strong Feller property if $P_t\bone_E \in C_b(E)$ for any $t>0$.
\end{thm}

\begin{proof}
For the moment, we fix $f\in {\mathscr B}_b(E) \cap L^p(E;\,\m)$ and $t>0$.  For $x \in E$, we define
\[
\widehat{T}_{t} f (x):= \frac{1}{2\pi {\rm i}} \int _{\Gamma _{\theta, \varepsilon}} e^{z t} \overline{Rf} (z,x) \,\d z,
\]
where $\overline{Rf}$ is the function obtained in Proposition~\ref{prop:rstFeller}.
Note that the line integral is well-defined by Proposition~\ref{prop:rstFeller}~\ref{prop:rstFeller3}. Let $K$ be a compact subset of $E$. 

Then, by \eqref{eq:unkn1}, there exists an $\m$-null subset $N_0$ of $E$ such that 
\begin{align*}
\sup_{x,y\in K\setminus N_0}
\int_{\Gamma_{\theta,\varepsilon}}e^{t\,\text{Re\,z}} \left|\overline{Rf} (z,x)
-\overline{Rf} (z,y)\right| \,|\d z|<\infty.
\end{align*}
Let $N$ be an $\m$-null set constructed in Proposition~\ref{prop:rstFeller}. 
We may assume $N_0\subset N$. 
Then, for any $\delta>0$, there exists a compact subset $L \subset \C$ such that 
\begin{align*}
&\frac{1}{2\pi}\sup_{x,y\in K\setminus N}\int_{\Gamma_{\theta,\varepsilon}}e^{t\,\text{Re\,z}} \left|\overline{Rf} (z,x)
-\overline{Rf} (z,y)\right| \,|\d z|\notag\\
&\leq \frac{1}{2\pi}
\sup_{x,y\in K\setminus N}\int_{\Gamma_{\theta,\varepsilon}\cap L}e^{t\,\text{Re\,z}} \left|\overline{Rf} (z,x)
-\overline{Rf} (z,y)\right| \,|\d z|+\delta.
\end{align*}
For $x,y \in K\setminus N$ and $j \in \N$,
\begin{align}
&\left| \widehat{T}_{t} f (x) -\widehat{T}_{t} f (y) \right| \le \frac{1}{2\pi } \int _{\Gamma _{\theta, \varepsilon}} e^{t {\rm Re}\,z} \left|\overline{Rf} (z,x) - \overline{Rf} (z,y)\right| |\d z| \notag \\
&\hspace{0.5cm}\le \frac{1}{2\pi }\sup_{x,y\in K\setminus N} \int _{\Gamma _{\theta, \varepsilon}\cap L} e^{t {\rm Re}\,z} \left|\overline{Rf} (z,x) - \overline{Rf} (z,y)\right| |\d z| +\delta \label{eq:j} \\
&\hspace{0.5cm}\le \frac{\delta|L|}{2\pi}\times 
e^{t\varepsilon}
+\frac{|L \cap \Gamma_{\theta,\varepsilon,\delta,j}|}{\pi}
 \times 
 e^{t\varepsilon}
 \times
 \left\|\overline{Rf}(\cdot,\cdot)\right\|_{L^\infty((L \cap \Gamma_{\theta,\varepsilon})\times K;\, |\d z|\otimes \m)} +\delta, \notag
\end{align}
where
$\Gamma_{\theta,\varepsilon,\delta,j}=\{z  \in \Gamma _{\theta, \varepsilon} \;|\; \sup_{x,y \in K \setminus N,\, d(x,y)<1/j} |\overline{Rf} (z,x) - \overline{Rf} (z,y)|>\delta \}.$
By letting $j \to \infty$ in \eqref{eq:j}, the Markov inequality, \eqref{eq:unkn2},  and Lebesgue convergence theorem together lead us to 
\[
\varlimsup_{j \to \infty}\sup_{x,y \in K \setminus N,\, d(x,y)<1/j}\left| \widehat{T}_{t} f (x) -\widehat{T}_{t} f (y) \right| \le \frac{\delta |L|}{2\pi}
\times e^{t\varepsilon}+\delta.
\]
Since $\delta$ is arbitrarily chosen, we see that $\widehat{T}_{t}f$ is uniformly continuous on $K \setminus N$. Because $K \setminus N$ is a dense subset of $K$, we find that $\widehat{T}_{t}f$ is extended to  a continuous function on $K$, which is denoted by the same symbol. Since $K$ is arbitrary, $\widehat{T}_{t}f$ is continuous on $E$.
From Proposition~\ref{prop:anal} and Proposition~\ref{prop:rstFeller}~\ref{prop:rstFeller1}, 
\begin{equation}\label{eq:v}
P_{t}f(x) = T_{t} f(x) = \widehat{T}_{t} f(x) \quad \m\text{-a.e.}\ x \in E.
\end{equation}
From this, $\widehat{T}_tf$ is bounded on $E$. In particular, if $f \in  L^p(E;\m)\cap C_b(E)$, the absolute continuity  and the sample path (right-)continuity  imply
\[
P_tf=\lim_{s \to 0} P_{s+t} f=\lim_{s \to 0} P_{s}( P_{t} f)=\lim_{s \to 0}P_{s}(\widehat{T}_{t} f)=\widehat{T}_{t} f.
\]
This implies that $P_t$ maps any function in $L^p(E;\m)\cap C_b(E)$ to $C_b(E)$. Even if $f \in   L^p(E;\,\m)\cap \mathscr{B}_b(E) $, we already know $\widehat{T}_tf \in  L^p(E;\m)\cap C_b(E)$.  Therefore, by using the $C_b$-Feller property, \eqref{eq:v}, and the absolute continuity, we obtain that for any $t>0$,
\[
P_tf=P_{t/2}(P_{t/2}f)=P_{t/2}(\widehat{T}_{t/2}f ) \in C_b(E).
\]
This implies that $P_t$ maps any function in $ L^p(E;\m)\cap \mathscr{B}_b(E)$ to $C_b(E)$.

``In particular'' part is proved as follows. We take a non-decreasing sequence $\{f_n\}_{n=1}^\infty \subset L^p(E;\,\m)\cap {\mathscr B}_b(E)$ such that $\lim_{n \to \infty}f_n(x)=f(x)$ for any $x \in E$. By the non-negativity of $f$ and the monotone convergence theorem, it holds that
\[
P_t f(x) = \bE_x\left[\lim_{n \to \infty}f_n(X_t) \right] = \lim_{n \to \infty} \bE_x[f_n(X_t)] =\lim_{n \to \infty} P_t f_n(x), \quad x\in E.
\]
Hence, we have the lower semicontinuity of $P_t f$ from the continuity of $P_t f_n$, $n \in {\mathbb N}$. 
Since the function $\|f\|_{\infty}-f$ is also a nonnegative bounded Borel function, 
we have that $P_t (\| f\| _\infty - f)$ is lower semicontinuous on $E$.
From this fact and the assumption that $P_t\bone_E \in C_b(E)$ for any $t>0$,  we find that  $P_t f$ is upper semicontinuous. Thus $P_t$ maps any non-negative function in $\mathscr{B}_b(E)$ to $C_b(E)$. For general $f \in {\mathscr B}_b(E)$, decomposing $f$ into $f_+ := f\bone_{\{ f>0\}}$ and $f_- := -f\bone_{\{ f\le 0\}}$ and applying the result above to $f_+$ and $f_-$, we obtain the semigroup strong Feller property.
\end{proof}

In the case that $\{ T_t\}_{t \ge 0}$ is not an analytic semigroup on $L^p(E;\,\m)$, even if $\{ R_\alpha\}_{\alpha>0}$ has the strong Feller property, $\{ P_t\}_{t \ge 0}$ may not be strong Feller. To see this, let $\{P_t\}_{t \ge 0}$ be the semigroup of the space-time Brownian motion:
\begin{align*}
P_tf(x,\tau):=\bE_{x}^{(1)}\otimes \bE_{\tau}^{(2)}\left[f(B_t,t)\right]\quad f \in \mathscr{B}(\R^2),\, (t,x)\in \R^2,\,t>0.
\end{align*}
Here, $(\{B_t\}_{t \ge 0},\{\bP_x^{(1)}\}_{x \in \R})$ is a one-dimensional Brownian motion, and $P_{\tau}^{(2)}$ denotes the law of uniform motion to the right starting from $\tau \in \R$ with unit speed. 
It is known that the semigroup $\{ P_t\}_{t \ge 0}$  does not have the strong Feller property, but the associated resolvent  has the strong Feller property (see Remark 1.1 in \cite{KKT}).
Now we let  $p\in [1,+\infty )$ 
and $\m$ be the Lebesgue measure on ${\R}^2$, and see the fact that the semigroup $\{ T_t\}_{t \ge 0}$ on $L^p({\R}^2; \m)$ which is generated by the operator
\[
A= \frac 12 \frac{\partial ^2}{\partial x_1 ^2} + \frac{\partial}{\partial x_2}, \quad x=(x_1,x_2),
\]
is not an analytic semigroup. 
It is easy to see that $\m$ is the invariant measure of $\{ T_t\}_{t \ge 0}$, and $\{ T_t\}_{ t \ge 0}$ is the transition semigroup generated by $\{(B_t ,t)\}_{t \ge 0}$.
Let $a\in {\R}$ and
\[
f_n(x_1,x_2) := \left( \frac{p}{\pi n}\right) ^{1/p} \exp \left({\rm i}\, a x_2 - \frac{x_1^2 + x_2^2}n\right) , \quad (x_1,x_2) \in {\R}^2
\]
for $n\in {\mathbb N}$.
Then,
\[
\int _{{\mathbb R}^2}\left| f_n\right|^p \d\m = \frac{p}{\pi n} \left( \int _{\mathbb R} \exp \left(- \frac{p y^2}n\right) \d y\right) ^2  =1
\]
and
\begin{align*}
\int _{{\mathbb R}^2}&\left| Af_n - {\rm i}\, a f_n\right|^p \d\m \\
&= \frac{p}{\pi n}\int _{{\R}^2}\left| \left( \frac 1n +\frac{2x_1^2}{n^2} + \frac{2x_2}{n}\right) \exp \left({\rm i}\, a x_2 -\frac{x_1^2 + x_2^2}n\right) \right| ^p \m (\d x) \\
&= \frac{p}{\pi n^{p/2}} \int _{{\R}^2} \left( \frac 1{\sqrt n} +\frac{2\xi _1^2}{\sqrt{n}} + 2\xi _2\right) ^p \exp \left( -p (\xi_1^2 + \xi_2^2) \right) \m (\d\xi) \\
&\rightarrow 0, \quad n\rightarrow \infty .
\end{align*}
These imply that ${\rm i}\, a \not \in \rho (A)$ for all $a\in {\R}$.
Therefore, $A$ is not a sectorial operator, and equivalently, $\{ T_t\}_{t \ge 0}$ is not analytic on $L^p({\R}^2;\m)$.

\section{Application to Markov processes associated with lower bounded semi-Dirichlet forms}\label{sec:appSemiDir}

Throughout this section, we assume that 
$(E,d)$ is a locally compact separable metric space with its one point compactification $E_{\partial}$ and $\m$ is a positive Radon measure on $E$ with full support, and 
${\bf (A2)}$ is satisfied. 
By Theorem~\ref{thm:anal} b),
for any $p \in [2,+\infty)$, 
the semigroup 
$\{T_t\}_{t\geq0}$
is analytic on $L^p(E;\,\m)$ with a sector $S_{\delta'}$ for some $\delta' \in (0,\pi/2]$.
Therefore, 
the generator $(A,\text{Dom}(A))$ on $L^p(E;\,\m)$ is also sectorial with some constants $\theta_p \in (\pi/2,\pi)$ and $M_p \in (0,+\infty)$ depending on $p$. See \cite[Theorem~1.4.2]{Da0} for a quantitative bound of $\theta_p$.

As noted before, the semigroup $P_tu$ with $u\in L^2(E;\,\m)\cap\mathscr{B}_b(E)$ is a quasi-continuous $\m$-version of 
 $T_tu$ for $t>0$.   
On the basis of this fact, we follow 
 \cite[\S3.5, Theorem~3.5.4]{Oshima} to obtain the next proposition.

\begin{prop}\label{prop:density}
If the resolvent kernel of ${\bf X}$ is absolutely continuous with respect to $\m$, so is the transition kernel,
\[
P_t(x,\d y)=P_t(x,y)\,\m(\d y),\quad t>0,\, x \in E.
\]
\end{prop}

Now, we apply Theorem~\ref{thm:equi} to 
${\bf X}$ under ${\bf (A2)}$ for $p\in[2,+\infty)$. In the sequel, we write $R_\alpha(x,y)$ for the resolvent kernel $R_\alpha(x,\cdot)$ of ${\bf X}$ with respect to $\m$ (if it exists). 
\begin{thm}\label{thm:app}
Suppose that ${\bf (A2)}$ is satisfied.
Assume  that $\{ R_\alpha\}_{\alpha >0}$ has the strong Feller property, and there exist $C>0$, $\alpha>\alpha_0$, and $p>2$ such that
\begin{equation}\label{eq:thmsym10}
\| G_\alpha g \|_{L^\infty(E;\,\m)} \leq C(\| g\| _{L^p(E;\,\m)}+
\| g\| _{L^2(E;\,\m)}),\quad g\in L^p(E;\,\m)\cap L^2(E;\,\m).
\end{equation}
In addition, we assume either of the following conditions is satisfied:
\begin{enumerate}
\item[{\rm(i)}] $\m(E)$ is finite, 
\item[{\rm(ii)}] for any $t>0$, $P_t\bone_{E}$ is continuous on $E$.
\end{enumerate}
Then, the semigroup of ${\bf X}$ has the strong Feller property.
\end{thm}

\begin{proof}
In view of Theorem~\ref{thm:equi} with $p=2$, Lemma~\ref{lem:density}, and  Proposition~\ref{prop:density}, it suffices to confirm the conditions in Proposition~\ref{prop:rstFeller} with $p=2$, 
\eqref{eq:unkn1} and \eqref{eq:unkn2}. 

We take positive constants $C>0$, $\alpha>\alpha_0$, and $p>2$ so that \eqref{eq:thmsym10}  holds. On both $L^p(E;\,\m)$ and $L^2(E;\,\m)$, we see that $(A,\text{Dom}(A))$ is a sectorial operator with constants $\theta=\theta_2\land \theta_p \in (\pi/2,\pi)$ and $M=M_2\lor M_p$.  We fix $z \in S_\theta$ and $f\in L^2(E;\,\m)\cap  C_b(E)$.  Then, we have $f\in L^p(E;\,\m)\cap  C_b(E)$. From the resolvent equation and \eqref{eq:thmsym10}, 
\begin{align}
\| R(z\,;A)f \|_{L^\infty(E;\,\m)}&= \| R(\alpha\,;A) f+ (\alpha -z) R(\alpha\,;A) R(z\,;A) f\|_{L^\infty(E;\,\m)}  \notag \\
 &= \| G_\alpha f+ (\alpha -z)G_\alpha R(z\,;A) f\|_{L^\infty(E;\,\m)}  \notag \\
&\le  \| G_\alpha f\|_{L^\infty(E;\,\m)}+ \|(\alpha -z) G_\alpha R(z\,;A) f\|_{L^\infty(E;\,\m)} \notag \\
&\le C(\| f\|_{L^p(E;\,\m)}+\| f\|_{L^2(E;\,\m)})\notag \\
&\quad +C|\alpha-z| (\|R(z\,;A) f\|_{L^p(E;\,\m)}+\|R(z\,;A) f\|_{L^2(E;\,\m)})\notag \\
&\le  C(\| f\|_{L^p(E;\,\m)}+\| f\|_{L^2(E;\,\m)})\label{eq:thmsym40} \\
&\quad +C M\frac{|\alpha-z|}{|z|}(\|f\|_{L^p(E;\,\m)}+\|f\|_{L^2(E;\,\m)})<\infty. \notag
\end{align} 
By using the resolvent equation again, we have
\begin{equation}\label{eq:thmsym30}
R(z\,;A) f = R(\alpha \,;A) f + (\alpha -z) R(\alpha \,;A)  R(z\,;A) f \quad \text{ in }\quad  L^2(E;\,\m).
\end{equation}
Then, \eqref{eq:thmsym40} and \eqref{eq:thmsym30} imply that $R(z\,;A)f$ possesses a continuous $\m$-version on $E$. Because $\Gamma_{\theta,\varepsilon}\subset S_{\theta}$ for any $\varepsilon \in (0,\theta-\pi/2)$, the condition~(ii) in Proposition~\ref{prop:rstFeller} holds. The conditions \eqref{eq:unkn1} and \eqref{eq:unkn2} immediately follow from \eqref{eq:thmsym40}.
\end{proof}
\begin{rem}\label{rem:sobolev}
{\rm 
\begin{itemize}
\item[{\rm (i)}] Assume that ${\bf (A2)}$ is satisfied. 
Then, \cite[Corollary]{Fuk1977} (see also \cite[p.~98]{Oshima})  implies that \eqref{eq:thmsym10} holds for $p>(q/(q-2))\lor 2$ and $\alpha>\alpha_0$
if there exist $q>2$, $\alpha \ge \alpha_0$ and $S>0$ such that $L^q(E;\,\m) \subset \mathscr{F}$ and   
\begin{align}\label{eq:quc}
\|f\|_{L^q(E;\,\m)}^{2} \le S\mathscr{E}_{\alpha}(f,f),\quad f \in \mathscr{F}.
\end{align}
The inequality~\eqref{eq:quc} is often called the \emph{Sobolev type inequality}. 
Next we assume that $(\mathscr{E},\mathscr{F})$ is a symmetric Dirichlet form on $L^2(E;\,\m)$. Then 
\eqref{eq:quc} implies that the associated semigroup  $\{T_t\}_{t>0}$ is ultracontractive, i.e.,  each $T_t$ is extended to a bounded linear operator from $L^1(E;\m)$ to $L^\infty(E;\,\m)$, and more strongly, we have 
\begin{align}\label{eq:suc}
\|T_t\|_{1 \to \infty} \le Ct^{-q/(q-2)}e^{\alpha t},\quad t>0
\end{align}
for some positive constant $C>0$.  We also see from  \cite[Lemma~2.1.2]{Da} that for any $t>0$, $\|T_{t/2}\|_{2 \to \infty}^2=\|T_t\|_{1 \to \infty}$ under the symmetry of $\{T_t\}_{t>0}$.  It is also known that the ultracontractivity of the semigroup and \eqref{eq:suc} together imply \eqref{eq:quc}.
See \cite[Theorems~(2.1), (2.9), and  (2.16)]{CKS},  \cite[Theorem~4.2.7]{FOT}, or
\cite[Theorem~1]{Var} for the proof.

\item[{\rm (ii)}] 
When $(\mathscr{E},\mathscr{F})$ is a 
symmetric Dirichlet form on $L^2(E;\,\m)$,
the following condition implies \eqref{eq:quc} by \cite[Theorem~4.1(i)]{MoriLpKato}:
\begin{align}
\sup _{x\in E} \int _E R_\alpha (x,y)^q \,\m(\d y) <\infty\label{eq:quc*}
\end{align}
 for some $\alpha\in(0+\infty)$ and $q\in(1,+\infty)$. 
We also see that \eqref{eq:quc*} for such $\alpha,q$ yields that  
there exists $C\in (0,+\infty)$ such that  $\|G_\alpha f\|_{L^\infty(E;\,\m)} \le C \|f\|_{L^p(E;\,\m)}$ for any $f \in L^p(E;\,\m)$ with $p=q/(q-1)$.
\end{itemize}
}
\end{rem}

Even if $(\mathscr{E},\mathscr{F})$ is a symmetric Dirichlet form, it cannot be expected that \eqref{eq:thmsym10} follows from the condition that the associated semigroup  $\{T_t\}_{t>0}$ is  ultracontractive only. In fact, in this case, we do not know a quantitative estimate of $ \|T_t\|_{1 \to \infty} (= \|T_{t/2}\|_{2 \to \infty}^2)$ in $t>0$.  However, under the ultracontractivity, we can improve the proof of \cite[Proposition~3.4]{BKK} to obtain the same result as Theorem~\ref{thm:app}.

\begin{thm}\label{thm:bkk}
Suppose that ${\bf (A2)}$ is satisfied.
Assume  that $\{ R_\alpha\}_{\alpha >0}$ has the strong Feller property, and $\|T_t\|_{L^2\to L^\infty}$ is finite for any $t>0$.
In addition, we assume either of the following conditions is satisfied:
\begin{enumerate}
\item[{\rm(i)}] $\m(E)$ is finite, 
\item[{\rm(ii)}] for any $t>0$, $P_t\bone_{E}$ is continuous on $E$.
\end{enumerate}
Then, the semigroup of ${\bf X}$ has the strong Feller property. 
\end{thm}
\begin{proof}
We fix $t>0$.  
By the duality, we have $\|T_t^*\|_{1 \to 2}=\|T_t\|_{2 \to \infty}.$ That is, $T_t^*$ is extended to a bounded linear operator from $L^1(E;\,\m)$ to $L^2(E;\,\m)$. Fix $f \in L^2(E;\,\m) \cap \mathscr{B}_b(E)$, and let
\begin{equation}\label{eq:hh}
h=(\alpha-A)T_t f .
\end{equation}
Here, $\alpha$ is a positive number with $\alpha>\alpha_0$. We see from \cite[Chapter~I, Exercise~1.9]{MR} that $T_tf \in \text{Dom}(A)$. Therefore, $h \in \text{Dom}(A) \subset L^2(E;\,\m)$ and 
\begin{equation}\label{eq:an1}
A T_t f=AT_{t/2}T_{t/2}f=T_{t/2}(AT_{t/2}f).
\end{equation}
 By noting that $\{e^{-\alpha_0 t }T_t\}_{ t\ge 0}$ is an analytic semigroup on $L^2(E;\,\m)$, we have from \cite[Chapter~2, Theorem~5.2~(d)]{Pa} that 
 \begin{equation}\label{eq:an2}
 \|AT_{t/2}f\|_{L^2(E;\,\m)} \le (C/t)\|f\|_{L^2(E;\,\m)}
 \end{equation}
 for some $C$ independent of $t$ and $f$. Combining \eqref{eq:an1} and \eqref{eq:an2}, we obtain that for any $g \in L^1(E;\,\m)$,

\begin{align*}
&\left|\int_{E} hg\,\d \m\right|=\left|((\alpha-A)T_tf,g)_{L^2(E;\,\m)}\right| 
=\left|((\alpha-A)T_{t/2}f,T_{t/2}^*g)_{L^2(E;\,\m)} \right|
\\
&\hspace{0.5cm}\le \left(\alpha e^{\alpha_0t/2}\|f\|_{L^2(E;\,\m)}
+\|AT_{t/2}f\|_{L^2(E;\,\m)}\right)\|T_{t/2}^*g\|_{L^2(E;\,\m)} \\
&\hspace{0.5cm}\le 
\left(\alpha e^{\alpha_0t/2}
+\frac{C}{t}\right)\|f\|_{L^2(E;\,\m)}
\|T_{t/2}^\ast g\|_{L^2(E;\,\m)} \\
&\hspace{0.5cm}\le  \left(\alpha e^{\alpha_0t/2}
+\frac{C}{t}\right)\|f\|_{L^2(E;\,\m)}\|T_{t/2}^\ast \|_{L^1 \to L^2}\|g\|_{L^1(E;\,\m)}.
 \end{align*} 
 This shows that the functional $L^1(E;\,\m) \ni g \mapsto \int_{E} hg\,\d \m$ is continuous. Therefore, we find that $h$ belongs to $L^2(E;\,\m) \cap L^\infty(E;\,\m)$, and obtain
 \begin{equation}\label{eq:bkk}
 P_{t}f=R_\alpha h.
 \end{equation}
From this and the resolvent strong Feller property, we find that $P_tf$ possesses a bounded continuous $\m$-version. Then, by following the same argument after \eqref{eq:v}, we know that $P_t$ maps any function in $ L^2(E;\m)\cap \mathscr{B}_b(E)$ to $C_b(E)$.  The rest of the proof is exactly similar to that of Theorem~\ref{thm:app}.
\end{proof}
\begin{rem}\label{rem:nsobolev}
{\rm 
\begin{itemize}
\item[(i)] Note that \eqref{eq:bkk} is obtained without the following assumptions: $E$ is locally compact, and  $\{T_t\}_{t\ge 0}$ is associated with a lower bounded semi-Dirichlet form. Consider a Markov process on a metric space with right continuous path and the resolvent strong Feller property. Then, if the semigroup is extended to an $L^2$-space with respect to a suitable measure and  has the ultracontractivity, we get the  same bounded function as \eqref{eq:hh}, which leads us to \eqref{eq:bkk}. Therefore, we obtain the semigroup strong Feller property if either of the same conditions as (i) and (ii) in Theorem~\ref{thm:bkk} holds.
\item[(ii)] 
A similar equation to \eqref{eq:bkk} is also obtained in the proof of \cite[Proposition~3.4]{BKK}, where the spectral decomposition theorem is used.
\item[(iii)] Let $(\mathscr{E},\mathscr{F})$ be a lower bounded semi-Dirichlet form on $L^2(E;\,\m)$ which satisfies \eqref{eq:quc} for some positive constants $q>2$, $\alpha \geq\alpha_0$ and $S>0$. Then, for any $f \in L^2(E;\,\m)$ and $t>0$, we have $U_tf:=e^{-\alpha_0 t}T_t f \in \text{Dom}(A) \subset \mathscr{F}$. The analyticity of $\{U_t\}_{t >0}$ implies that for any $t>0$ and $f \in L^2(E;\,\m)$,
\[
\|U_t f\|_{L^q(E;\,\m)} \le S\mathscr{E}_{\alpha}(U_tf,U_t f)=-S (AU_tf, U_tf)_{L^2(E;\,\m)}  \le \frac{C}{t} \|f\|_{L^2(E;\,\m)}
\]
for some $C>0$. Then,  applying \cite[Lemma~6.1]{Ou} to $\{U_t\}_{t \ge 0}$, we find that  $\{T_t\}_{t \ge 0}$ is extended to a bounded linear operator from $L^2(E;\,\m)$ to $L^\infty(E;\,m)$.
\end{itemize}}
\end{rem}

 Next, we localize conditions in Theorems~\ref{thm:app} and \ref{thm:bkk} in an appropriate framework. Then, we introduce subprocesses of ${\bf X}$. For an open subset $U \subset E$, we set $\tau_U=\inf\{t \in [0,+\infty) \mid X_t \notin U\}$. Then,  the  subprocess of ${\bf X}$ killed upon leaving $U$ is defined by 
\[
X_t^U:=
\begin{cases}X_t,\quad &\text{if }t<\tau_U,\\
\partial,\quad&\text{if } t \ge \tau_U.
\end{cases}
\]
We see from \cite[\S3.5, Theorem~3.5.7]{Oshima}
that ${\bf X}^U=(\{X_t^U\}_{t \in [0,+\infty]},\{\bP_x\}_{x \in U})$ 
is associated with the lower bounded semi-Dirichlet form $(\mathscr{E}_U,\mathscr{F}_U)$ on $L^2(U;\,\m)$, and it 
is a Hunt process on $U$. We call ${\bf X}^U$ the \emph{part process of ${\bf X}$ on $U$.}
Here, $\mathscr{F}_U$ is identified with the completion of $\{u \in C_0(E) \cap \mathscr{F} \mid \text{supp}[u] \subset U\}$ with respect to $\mathscr{E}^{1/2}_{\alpha_{0}+1}$, and
$\mathscr{E}_U(u,v):=\mathscr{E}(u,v)$ for $u,v\in\mathscr{F}_U$. 
It is also proved in \cite[\S3.5, Theorem~3.5.7]{Oshima} that  
$(\mathscr{E}_U,\mathscr{F}_U)$ is a regular semi-Dirichlet form on $L^2(U;\,\m)$ having the same lower bound $-\alpha_0$ on $L^2(U;\,\m)$. 
Therefore, by Proposition~\ref{prop:L2semigroup},
the semigroup $\{P_t^U\}_{t \ge 0}$ and the resolvent $\{R_\alpha^U\}_{\alpha>0}$ are extended to bounded linear operators on $L^p(E;\,\m)$, $p \in [2,+\infty)$. The extensions are denote by $\{T_t^U\}_{t \ge 0}$ and $\{G_\alpha^U\}_{\alpha>0}$, respectively.  Furthermore,  
Theorem~\ref{thm:anal}
implies that $\{T_t^U\}_{t \ge 0}$ is analytic on $L^p(U;\m)$ for $p\in[2,+\infty)$. For $t,\alpha \in (0,+\infty)$, $x \in U$, $f \in \mathscr{B}_b(U)$, we have
\begin{equation}
P_t^Uf(x)=\bE_x[f(X_t):t<\tau_U],\quad
 R_\alpha^Uf(x)=\bE_x\left[\int_{0}^{\tau_U} e^{-\alpha t}f(X_t)\,\d t\right]. \label{eq:pr}
\end{equation}
Therefore, if the resolvent kernels of ${\bf X}$ is  absolutely continuous with respect to $\m$, so is the resolvent kernel of ${\bf X}^U$,
\begin{equation*}
R_\alpha^U(x,\d y)=R_\alpha^U(x,y)\,\m(\d y),\quad \alpha>0,\,x \in U.
\end{equation*}

The following theorem provides a sufficient condition for the resolvent strong Feller property of part processes of ${\bf X}$.
\begin{thm}[{\cite[Theorem~3.1]{KKT}}]\label{thm:kkt}
Assume that the resolvent of ${\bf X}$ has both the strong Feller property and the Feller property. Then,  for any open subset $U \subset E$, the resolvent of ${\bf X}^U$ has the strong Feller property.
\end{thm}

Let $U$ be an open subset of $E$. A bounded Borel measurable function $h\colon U \to \R$ is said to be \emph{harmonic} (with respect to ${\bf X}$) if for any relatively compact open set $V \subset U$, $\{h(X_{t \wedge \tau_V})\}_{t \ge 0}$ is a uniformly integrable martingale with respect to $\bP_x$, $x \in V$.  A 
Hunt process ${\bf X}$ is said to be a \emph{diffusion process without killing inside} if ${\bf X}$ is of continuous sample paths and $\bP_x(X_{\zeta-} \in E,\,\zeta<+\infty)=0$ for every $x \in E$. Here $\zeta$ denotes the lifetime of ${\bf X}$. With these definitions, we give another sufficient condition for the resolvent strong Feller property for part processes of ${\bf X}$.

\begin{prop}\label{prop:harm}
Suppose that ${\bf X}$ is a diffusion process without killing inside, and the resolvent has the strong Feller property. In addition, we assume that any bounded harmonic function $h$ on an open subset $U \subset E$ 
is continuous there.
Then, the resolvent of ${\bf X}^U$ has the strong Feller property.
\end{prop}
\begin{proof}
We fix $\alpha>0$, $f \in \mathscr{B}_b(U)$, and set $f=0$ on $E\setminus U$. In view of \eqref{eq:pr}, it suffices to show that  $\phi_U^\alpha:=\bE_{{}_{\bullet}}[\int_{\tau_U}^\infty e^{-\alpha t}f(X_t)\,\d t ]$ is continuous on $K$  for any compact subset $K \subset U$. 

We fix a compact subset $K \subset U$, and let $V$ be a relatively compact open subset of $U$ such that $K \subset V \subset \overline{V} \subset U$. Here we denote by $\overline{V}$ the closure of $V$ in $E$. For $n \in \N$ and $x \in U$, we define \[
\psi_U^n(x)=\bE_x \left[e^{-n \tau_U} \right].\] 
The assumption ``no killing inside'' ensures that $\tau_V<\tau_U$ for $\bP_x$-a.s. $x \in V$. It follows that
$
\tau_U=\tau_V+\tau_U \circ \theta_{\tau_V} \ge \tau_U \circ \theta_{\tau_V}$ for $\bP_x$-a.s. $x \in V.$
Here, $\{\theta_t\}_{t \in [0,+\infty]}$ denotes the shift operators of $X$. Then, by using the strong Markov property \cite[Theorem~A.1.21]{CF} of ${\bf X}$,  we have for any $x \in V$ and $n \in \N$,
\begin{equation}\label{eq:h}
\psi_U^n(x)=\bE_x \left[e^{-n \tau_U}:\tau_V<\tau_U \right]
\le \bE_x\left[\bE_{X_{\tau_V}}\left[e^{-n \tau_U}\right]\right]=:h_n (x).
\end{equation}
We see from the strong Markov property and the same argument as in the proof of \cite[Lemma~6.1.5]{CF} that $h_n$ is a harmonic function on $V$ with respect to ${\bf X}$. Noting this fact and the assumption that $h_n$ is bounded continuous on $K$, we use Dini's theorem to obtain that $\lim_{n \to \infty}\sup_{x \in K}h_n(x)=0$. In particular, we have from \eqref{eq:h} that
\begin{equation}\label{eq:h2}
\lim_{n \to \infty}\sup_{x \in K}\psi_U^n(x)=0.
\end{equation} 
The same argument as in \cite[Theorem~3.1]{KKT} and \eqref{eq:h2} imply
\[
\lim_{n \to \infty }\sup_{x \in K}\left| \phi_U^\alpha-nR_{n+\alpha}\phi_U^{\alpha}(x)\right| \le 2 \sup_{x \in U}|f(x)|  \times \lim_{n \to \infty }\sup_{x \in K}\psi_U^n(x) =0.
\]
Finally, by using the strong Feller property, we see $\phi_U^\alpha$ is continuous on $U.$
\end{proof}

\begin{rem}
{\rm By assuming both the Feller and the strong Feller property of the resolvent, we have the continuity of bounded harmonic functions. The proof is almost the same as \cite[Theorem~3.4]{SW}.
}
\end{rem}

 Let $U$ be an open subset of $E$. Then,  under both the situations of Theorem~\ref{thm:kkt} and Proposition~\ref{prop:harm}, we see that for any compact subset $K \subset U$,
\begin{equation}\label{eq:wf}
\lim_{s \to 0}\sup_{x \in K}\bP_x[\tau_U \le s]=0.
\end{equation}
In fact, under the situation of Theorem~\ref{thm:kkt},  we use \cite[Lemma~2.2]{KKT} to see that \eqref{eq:wf} is valid. For the latter situation, we take an open subset $V \subset E$ such that $K \subset V \subset \overline{V} \subset U$.  
It then holds that
$
\tau_U \ge \tau_U \circ \theta_{\tau_V},$ $\bP_x$-a.s. $x \in V.$ This and  the strong Markov property \cite[Theorem~A.1.21]{CF} of ${\bf X}$ together imply
\[
\bP_x[\tau_U \le s] \le \bE_x\left[\bP_{X_{\tau_V}}[\tau_U \le s] \right]=:u_s(x),\quad x \in V,\,s>0.
\]
Since $u_s$ is harmonic on $V$ with respect to ${\bf X}$, it is continuous on $V$. Then, it is straightforward to see that $\lim_{s \to 0}\sup_{x \in K}\bP_x[\tau_U \le s] \le \lim_{s \to 0}\sup_{x \in K}u_s(x)=0.$

In what follows, we say that ${\bf X}$ has the {\it local ultracontracitivity} if either of the following conditions is satisfied: 
\begin{itemize}
\item[(a)] for any relatively compact open set $U \subset E$, there exist $C>0$, $\alpha>\alpha_0$, and $p>2$ such that for any $g\in L^p(U;\,\m)\cap L^2(U;\,\m)$,
\begin{align}
\hspace{-1.5cm}\| G_\alpha^U g \|_{L^\infty(U;\,\m)} \leq C(\| g\| _{L^p(U;\,\m)}+
\| g\| _{L^2(U;\,\m)}).\label{eq:UCRW}
\end{align}
\item[(b)]  for any relatively compact open set $U \subset E$ and $t>0$, $\|T_t^U\|_{L^2(U;\,\m) \to L^\infty(U;\,\m)}$ is finite.
\end{itemize}
The condition \eqref{eq:UCRW} is weaker than $\|G_{\alpha}^U\|_{L^2(U;\,\m) \to L^\infty(U;\,\m)}<\infty$.
The localized version of Theorems~\ref{thm:app} and \ref{thm:bkk} are as follows:
\begin{thm}\label{thm:loc}
Assume that the resolvent of ${\bf X}$ has the strong Feller property, and the local ultracontractivity. In addition, we assume either of the following conditions is satisfied:
\begin{enumerate}
\item[{\rm(i)}] the resolvent of ${\bf X}$ has the Feller property,
\item[{\rm(ii)}] ${\bf X}$ is a diffusion process without killing inside, and for any relatively compact open subset $U \subset E$, any bounded harmonic function on $U$ is continuous there.
\end{enumerate}
Then, the semigroup of ${\bf X}$ has the strong Feller property.
\end{thm}
\begin{proof}
Let $U$ be a relatively compact open subset of $E$. From Theorem~\ref{thm:kkt} and Proposition~\ref{prop:harm}, it follows that the resolvent of ${\bf X}^U$ has the strong Feller property. Furthermore, Theorems~\ref{thm:app} and \ref{thm:bkk}, and the local ultracontractivity together imply that the semigroup of ${\bf X}^U$ has the strong Feller property. Let $K$ be a compact subset of $U$. We obtain from \eqref{eq:wf} that for any $t>0$ and $f \in \mathscr{B}_b(E)$, 
\[
\lim_{s \to 0}\sup_{x \in K}\left|P_tf(x)-P_s^UP_{t-s}f(x)\right|\le \sup_{x \in E}|f(x)|\times  \lim_{s \to 0}\sup_{x \in K}\bP_x[\tau_U \le s]=0.
\]
Thus, $P_tf$ is a continuous function on $K$. Because $K$ and $U$ are arbitrarily chosen, we see $P_tf$ is  continuous on $E$.
\end{proof}

In the proof of Theorem~\ref{thm:loc}, after establishing the semigroup strong Feller property of the part process, \eqref{eq:wf} is used. However, this can be replaced by the condition that $P_t \bone_{E}$ is continuous on $E$ for any $t>0$.  To clarify this fact, and for future reference, we prove the following lemma.

\begin{lem}
The semigroup of ${\bf X}$ has the strong Feller property if the following conditions are satisfied:
\begin{itemize}
\item for any relatively compact open subset $U \subset E$, the semigroup of ${\bf X}^U$ is strong Feller.
\item for any $t>0$, $P_t \bone_E$ is continuous on $E$.
\end{itemize}
\end{lem}
\begin{proof}
We follow the argument in \cite[Theorem~1.4]{ChKu}.
Fix  $t>0$ and a compact subset $K$ of $E$. Let $\{U_n\}_{n=1}^\infty $ be a sequence of relatively compact open subsets such that  $K \subset U_1$ and  $\overline{U}_n \subset U_{n+1}$ for any $n \in \N$. The assumptions imply that for any $n \in \N$,
\[
x \mapsto {\bf P}_x[\tau_{U_n} \le t<\zeta](=P_t\bone_{E}(x)-P_t^{U_n}\bone_{E}(x))
\]
is continuous on $K$. Thus, there exists $\{x_n\}_{n=1}^\infty \subset K$ such that
\[
{\bf P}_{x_n}[\tau_{U_n} \le t<\zeta]=\sup_{x \in K}\{P_t\bone_{E}(x)-P_t^{U_n}\bone_{E}(x)\},\quad n \in \N.
\]
Because $K$ is compact, there exists a subsequence of $\{x_n\}_{n=1}^\infty$  which converges to some $x \in K.$ We denote the subsequence $\{x_n\}_{n=1}^\infty$ again. For any $n,m \in \N$ with $n>m$, 
\[
{\bf P}_{x_n}[\tau_{U_n} \le t<\zeta] \le {\bf P}_{x_n}[\tau_{U_m} \le t<\zeta].
\]
By using the continuity of the map $x \mapsto {\bf P}_{x}[\tau_{U_m} \le t<\zeta]$, we obtain that for any $m \in \N$,
\[
\varlimsup_{n \to \infty}{\bf P}_{x_n}[\tau_{U_n} \le t<\zeta] \le {\bf P}_{x}[\tau_{U_m} \le t<\zeta].\]
Thus, we arrive at
\[\varlimsup_{n \to \infty}{\bf P}_{x_n}[\tau_{U_n} \le t<\zeta] \le \lim_{m \to \infty}{\bf P}_{x}[\tau_{U_m} \le t<\zeta]=0,\]
where we use the fact that quasi-left continuity up to $\zeta$ of ${\bf X}$ implies
${\bf P}_x(\lim_{n\to\infty}\tau_{U_n}=\zeta)=1$.
This shows that for any $f \in \mathscr{B}_b(E)$ and $t>0$, 
\[
\varlimsup_{n \to \infty}\sup_{x \in K}|P_tf(x)-P_t^{U_n}f(x)|\le \|f\|_{\infty} \times \varlimsup_{n \to \infty} \sup_{x \in K}{\bf P}_{x}[\tau_{U_n} \le t<\zeta]=0.
\]
Since the semigroups of $\{{\bf X}^{U_n}\}_{n=1}^\infty$ are strong Feller, so is the semigroup of ${\bf X}$.
\end{proof}

We close this section to provide some examples.
\begin{exam}
{\rm 
Let ${\bf B}=(\{B_t\}_{t \ge 0}, \{{\bf P}_{x}\}_{x \in \R})$ be a one-dimensional Brownian motion. For $t \ge 0,$ define $A_t=\int_{0}^t (1+|B_s|^4)^{-1}\,\d s$. Then, $A=\{A_t\}_{t \ge 0}$ is a positive continuous additive functional (PCAF) of ${\bf B}$, and the Revuz measure is identified with $\nu(\d x)=(1+|x|^4)^{-1}\,\d x$ (see \cite[Section~5.1]{FOT} for the definition of Revuz measures and the correspondence with PCAFs). Here, $\d x$ denotes the one-dimensional Lebesgue measure. Let $X_t=B_{A_t^{-1}}$, $t \ge 0$ be the time changed process by 
$A_t^{-1}:=\inf\{s>0\mid A_s>t\}$. By \cite[Theorem~6.2.1]{FOT}, ${\bf X}:=(\{X_t\}_{t \ge 0}, \{{\bf P}_{x}\}_{x \in \R})$  is a $\nu$-symmetric Hunt process  on $\R$, and the Dirichlet form associated with ${\bf X}$ is regular on $L^2(\R^d;\,\nu)$. It is straightforward to see that \[\lim_{t\to 0}\sup_{x \in \R}{\bf E}_x[A_t]=0.\]
In other words, $\nu$ is of Kato class with respect to ${\bf B}.$ Then,   \cite[Lemma~4.1]{KK} shows that the resolvent of  ${\bf X}$ is strong Feller. On the other hand, since the generator of ${\bf X}$ is 
$(1+x^4) (\d^2/\d x^2)$ $(x \in \R),$ the conditions~(iii) and (iv) in \cite[Theorem~8.4.1]{P} hold, which implies that the resolvent of ${\bf X}$ is not Feller. However, ${\bf X}$ satisfies condition {\rm (b)} in Theorem~\ref{thm:loc} since any harmonic function with respect to ${\bf X}$ is harmonic with respect to ${\bf B}$. Let $U$ be a bounded open interval, and let $T_U=\inf\{t>0 \mid B_t \notin U\}$. Noting $\tau_U:=\inf\{t>0 \mid X_t \notin U\}= A_{T_U}$ and using \cite[Proposition~4.1.10]{CF}, we obtain that for any $f \in L^2(\R;\,\nu)$, $x \in U$, and $\alpha>0$,
\begin{align*}
{\bf E}_x\left[\int_{0}^{\tau_{U}} e^{-\alpha t}|f(X_t)|\,\d t\right]\le {\bf E}_x\left[\int_{0}^{A_{T_U}} |f(X_t)|\,\d t\right]&={\bf E}_x\left[\int_{0}^{T_U} |f(B_t)|\,\d A_t\right]\\
&=\int_{U}|f(y)|g_U(x,y)\,\nu(\d y).
\end{align*}
Here, $g_U(x,y)$ denotes the green function of ${\bf B}^U$. Since $\sup_{x,y \in U}g_U(x,y)<\infty$, we see that ${\bf X}$ has the local ultracontractivity. Therefore, by Theorem~\ref{thm:loc}, the semigroup of ${\bf X}$ is strong Feller.
}
\end{exam}

\begin{exam}
{\rm 
Let ${\bf X}=(\{X_t\}_{t \ge 0}, \{{\bf P}_x\}_{x \in \R^d})$ be an Ornstein--Uhlenbeck process: 
\[
X_t=e^{-t/2}x+\int_{0}^t e^{(1/2)(t-s)}\,\d B_s,\quad t\ge 0,\, x \in \R^d.\]
Here, $B=\{B_t\}_{t \ge 0}$ is a $d$-dimensional Brownian motion starting at the origin. Define a Borel measure $\mu$ on $\R^d$ by $\mu(\d x)=\exp(-|x|^2/2)\,\d x$, where $\d x$ denotes the $d$-dimensional Lebesgue measure. ${\bf X}$ is a $\mu$-symmetric Hunt process. The Dirichlet form $(\mathscr{E},\mathscr{F})$ 
associated with ${\bf X}$ is regular on $L^2(\R^d;\,\mu)$, and the core is identified with $C_0^\infty(\R^d)$, the space of smooth functions on $\R^d$ with compact support. For $u,v \in C_0^\infty(\R^d)$, we have
\[
\mathscr{E}(u,v)=\frac{1}{2}\int_{\R^d}\langle \nabla u(x), \nabla v(x)\rangle \,\mu(\d x).
\]
Hereafter, $\nabla$ denotes the standard gradient and $\langle \cdot,\cdot\rangle$ the inner product on $\R^d$. The semigroup of ${\bf X}$ has both the strong Feller property and Feller property.
However, Theorems~\ref{thm:app} and \ref{thm:bkk} are not available for ${\bf X}$. Indeed, we have
\begin{align*}
{\bf E}_x\left[|X_t|^2\right]
&={\bf E}_x\left[|X_t-e^{-t/2}x|^2\right]+2{\bf E}_x\left[
\left\langle e^{-t/2}x,\int_{0}^t e^{(1/2)(t-s)}\,\d B_s\right\rangle
\right]+e^{-t}|x|^2\\
&
\ge 0+0+e^{-t}|x|^2,\quad x\in \R^d.
\end{align*}
Hereafter, we define $|\cdot|=\langle \cdot,\cdot\rangle^{1/2}$. Although the map $\R^d \ni x \mapsto |x|^2$ belongs to $L^p(\R^d;\,\mu)$ for any $p>0$, we see
\[
\sup_{x \in \R^d}{\bf E}_x[|X_t|^2]=\infty \quad \text{ and }\quad \sup_{x \in \R^d}{\bf E}_x\left[\int_{0}^\infty e^{-\alpha t}|X_t|^2\, \d t\right]=\infty
\]
for any $t>0$ and $\alpha>0$. This implies that both $\|T_t\|_{L^2\to L^{\infty}}<\infty$ and \eqref{eq:thmsym10} fail. On the other hand, ${\bf X}$ has the local ultracontractivity. To see this, let $U$ be an open ball in $\R^d$. Then, \cite[Theorem~4.4.3~(i)]{FOT} shows that the core of the Dirichlet form $(\mathscr{E}^U,\mathscr{F}^U)$ of $X^U$ is identified with $C_0^\infty(U)(=C_0^\infty(\R^d) \cap C_0(U))$ and
$
\mathscr{E}^U(u,u)=\mathscr{E}(u,u)$, $u\in C_0^\infty(U).
$
From \cite[Theorems~11.2, 11.23, and 11.34]{Leo}, there exists $C \in (0,+\infty)$ such that
\begin{align}\label{eq:cs}
\|u\|_{L^{q}(U;\,\d x)}^2 \le C\left(\int_{U}|\nabla u(x)|^2\,\d x +\|u\|_{L^{2}(U;\,\d x)}^2  \right),\quad u \in C_0^\infty(U).
\end{align}
Here, $q$ is given by
\[
q=\begin{cases}
2d/(d-2),&\quad \text{ if $d\ge 3$},\\
\text{any number in $(2,+\infty)$},&\quad  \text{ if $d=2$},\\
\infty,&\quad  \text{ if $d=1$}.
\end{cases}
\]
By \eqref{eq:cs} and the boundedness of $U$, there exists $C>0$ such that 
\begin{align}\label{eq:cs2}
\|u\|_{L^{q}(U;\,\mu)}^2\leq 
C \left(\mathscr{E}^U(u,u)+\|u\|_{L^{2}(U;\,\mu)}^2\right),\quad u \in C_0^\infty(U).
\end{align}
Since $C_0^\infty(U)$ is a core of  $(\mathscr{E}^U,\mathscr{F}^U)$, we see that \eqref{eq:cs2} is valid for any $u \in \mathscr{F}^U.$
This and Remark~\ref{rem:sobolev}~(i) imply  that ${\bf X}$ has the local ultracontractivity.
}
\end{exam}

The following example appears in \cite[\S1.5.2, (1.5.17)]{Oshima}.
\begin{exam}
{\rm  
Let $\underline{\alpha}, \overline{\alpha}, M$, and $\delta$ be positive numbers and let $\alpha\colon \R^d \to \R$ be a continuous function such that $
(1/2)(2\overline{\alpha}-\underline{\alpha})<\delta<1$, $ \overline{\alpha}<1+\underline{\alpha}/2$, 
and 
\[
|\alpha(x)-\alpha(y)|\le M|x-y|^{\delta},\quad 0<\underline{\alpha} \le \alpha(x) \le \overline{\alpha}<2,\quad x,y \in \R^d.
\]
Let $C^2_0(\R^d)$ the space of $C^2$-functions on $\R^d$ with compact support. We define 
\[
\mathscr{L}u(x)=\int_{\R^d}\{u(x+h)-u(x)-\langle \nabla u(x),h\bone_{B_1(0)}(h)\rangle \}\,\frac{w(x)}{|h|^{d+\alpha(x)}}\,dh,\quad u \in C^2_0(\R^d).
\]
Here, $B_1(0)$ denotes the open ball centered at the origin with radius $1$, and the weight function $w\colon \R^d \to \R$ is chosen so that $\mathscr{L}e^{-{\rm i}\langle h,x\rangle}=-|h|^{\alpha(x)} e^{-{\rm i}\langle h,x\rangle }$, $x,h \in \R^d$. Then, we have \[
w(x)=2^{\alpha(x)-1}\pi^{-(d/2)-1}\Gamma((1+\alpha(x))/2)\Gamma((\alpha(x)+d)/2)\sin(\pi \alpha(x)/2),\quad x \in \R^d.\]
From \cite[\S1.5.2, (1.5.18)]{Oshima}, $(\mathscr{L},C_0^2(\R^d))$ is associated with a lower bounded semi-Dirichlet form $(\mathscr{E},\mathscr{F})$ on $L^2(\R^d;\,\d x)$, which is described as
\[
\mathscr{E}(u,v)=\lim_{n \to \infty}\iint_{|x-y|>1/n}(u(x)-u(y))v(x)\frac{w(x)}{|x-y|^{d+\alpha(x)}}\,\d x\d y,\quad u,v \in C^1_0(\R^d).
\]
Here, $C^1_0(\R^d)$ denotes the space of $C^1$-functions on $\R^d$ with compact support.  Denote by ${\bf X}=(\{X_t\}_{t \ge 0}, \{{\bf P}_x\}_{x \in \R^d})$ the Hunt process associated with $(\mathscr{E},\mathscr{F})$. For $x,y \in \R^d$, we set $j(x,y)=w(x)/|x-y|^{d+\alpha(x)}$, and suppose that there exists $C>0$ such that \[j(x,y)\ge C|x-y|^{-d-1},\quad x,y \in \R^d\text{ with } 0<|x-y|<1.\] Then, by the argument in \cite[\S~3.5, Example~3.5.5]{Oshima}  and Remark~\ref{rem:nsobolev}(iii), the semigroup $\{T_t\}_{t>0}$ of $(\mathscr{E},\mathscr{F})$ satisfies $\|T_t\|_{L^2(\R^d;\,\d x) \to L^\infty(\R^d;\, \d x)}<\infty$ for any $t>0.$ Hence, the semigroup of ${\bf X}$ is strong Feller if the resolvent is. In \cite{BK0,BK}, the Harnack inequality for bounded harmonic functions on domains with respect to non-local operators with variable orders are obtained. Thus, by using \cite[Proposition~3.1]{BK0} and the argument in \cite[Proposition~3.3]{BKK}, we can also give a sufficient condition for the resolvent of {\bf X} being H\"{o}lder continuous in the spatial variable.
}
\end{exam}

The final example is due to \cite{Be}, which is a diffusion process on an infinite-dimensional space. 

\begin{exam}
{\rm  
Let $\T=\mathbb{R}/\mathbb{Z}$ be the one-dimensional torus, and denote by $\T^\infty$ the product of countably many copies of $\T$. That is, $\T^\infty$ is the infinite-dimensional torus. $\T^\infty$ becomes a compact space by the Tychonoff's theorem. We simply denote by $\d x$ the product measure on $\T^\infty$ of the normalized Haar measure on $\T$. Let $\mathscr{A}=\{a_k\}_{k=1}^\infty$ be a sequence of strictly positive numbers, and set
\[
\mathscr{E}^{\mathscr{A}}(u,v)=\int_{\T^\infty} \sum_{k=1}^\infty a_k \frac{\partial u}{\partial x_k}(x) \frac{\partial v}{\partial x_k}(x) \,\d x,\quad u,v \in \mathscr{D}.
\]
Hereafter, $\mathscr{D}$ denotes the set of cylindrical smooth functions on $\T^\infty$. 
It is shown in \cite[Section~1]{Be} that $(\mathscr{E}^{\mathscr{A}},\mathscr{D})$ is well-defined, and closable on $L^2(\T^\infty;\,\d x)$. Let $(\mathscr{E}^{\mathscr{A}},\mathscr{F}^{\mathscr{A}})$ be the smallest closed extension. Then, $(\mathscr{E}^{\mathscr{A}},\mathscr{F}^{\mathscr{A}})$ is a regular Dirichlet form on $L^2(\T^\infty;\, \d x).$ It is straightforward to see that $(\mathscr{E}^{\mathscr{A}},\mathscr{F}^{\mathscr{A}})$ is reccurrent. In particular, it is conservative. 
Under a suitable condition, we see from \cite[Lemma~7]{Be} that the associated semigroup $\{T_t^\mathscr{A}\}_{t>0}$ on $L^2(\T^\infty;\, \d x)$ possesses an integral kernel which is continuous on $(0,+\infty) \times \T^\infty \times \T^\infty$. From this fact, the conservativeness, and the compactness of the state space, $\{T_t^\mathscr{A}\}_{t>0}$ generates a Feller process with the semigroup strong Feller property. We also see that $\{T_t^\mathscr{A}\}_{t>0}$  is ultracontractive. 
However, \cite[Theorem~6]{Be} implies that the Sobolev type inequality \eqref{eq:quc} does not hold for $(\mathscr{E}^{\mathscr{A}},\mathscr{F}^{\mathscr{A}})$.
}
\end{exam}




\end{document}